\documentclass[11pt,oneside,reqno]{article}
\usepackage[a4paper,width=15cm,height=24cm]{geometry}

\usepackage{amsmath,amsthm,amssymb,color}
\usepackage[authoryear,longnamesfirst]{natbib}
\usepackage{array}
\usepackage{filemod}
\usepackage{bbm}
\usepackage{mathrsfs}
\usepackage[breaklinks=true]{hyperref}

\numberwithin{equation}{section}
\allowdisplaybreaks[4]

\theoremstyle{plain}
\newtheorem{theorem}{Theorem}[section]
\newtheorem{proposition}[theorem]{Proposition}
\newtheorem{lemma}[theorem]{Lemma}

\newtheorem{conjecture}[theorem]{Conjecture}

\theoremstyle{definition}
\newtheorem{question}[theorem]{Question}
\newtheorem{definition}[theorem]{Definition}
\newtheorem{example}[theorem]{Example}
\newtheorem{conjexample}[theorem]{Conjectural Example}
\newtheorem{remark}[theorem]{Remark}

\makeatletter

\renewcommand{\cite}{\citet}

\def\^#1{\ifmmode {\mathaccent"705E #1} \else {\accent94 #1} \fi}
\def\~#1{\ifmmode {\mathaccent"707E #1} \else {\accent"7E #1} \fi}

\def\*#1{#1^\ast}
\edef\-#1{\noexpand\ifmmode {\noexpand\bar{#1}} \noexpand\else \-#1\noexpand\fi}
\def\>#1{\vec{#1}}
\def\.#1{\dot{#1}}

\def\atop{\@@atop}
\def\%#1{\mathcal{#1}}

\renewcommand{\leq}{\leqslant}
\renewcommand{\geq}{\geqslant}
\renewcommand{\phi}{\varphi}
\newcommand{\eps}{\varepsilon}

\newcommand{\eq}{\eqref}

\newcommand{\dtv}{\mathop{d_{\mathrm{TV}}}}

\newcommand{\bigo}{\mathrm{O}}

\def\tsfrac#1#2{{\textstyle\frac{#1}{#2}}}
\newcommand{\toinf}{\to\infty}

\newcommand{\I}{\mathrm{I}}

\newcommand{\Exp}{\mathop{\mathrm{Exp}}}

\newcommand{\Beta}{{\mathrm{Beta}}}

\newcommand{\IE}{\mathbbm{E}}
\newcommand{\IP}{\mathbbm{P}}

\newcommand{\law}{\mathscr{L}}
\newcommand{\eqlaw}{\stackrel{\mathscr{D}}{=}}

\newcommand{\IR}{\mathbbm{R}}

\def\be#1{\begin{equation*}#1\end{equation*}}
\def\ben#1{\begin{equation}#1\end{equation}}
\def\bes#1{\begin{equation*}\begin{split}#1\end{split}\end{equation*}}
\def\besn#1{\begin{equation}\begin{split}#1\end{split}\end{equation}}

\def\ba#1{\begin{align*}#1\end{align*}}
\def\ban#1{\begin{align}#1\end{align}}

\def\klr#1{(#1)}
\def\bklr#1{\bigl(#1\bigr)}
\def\bbklr#1{\Bigl(#1\Bigr)}
\def\bbbklr#1{\biggl(#1\biggr)}
\def\bbbbklr#1{\Biggl(#1\Biggr)}

\def\kle#1{[#1]}
\def\bkle#1{\bigl[#1\bigr]}

\def\bbbkle#1{\biggl[#1\biggr]}

\def\bklg#1{\bigl\{#1\bigr\}}
\def\bbklg#1{\Bigl\{#1\Bigr\}}
\def\bbbklg#1{\biggl\{#1\biggr\}}
\def\bbbbklg#1{\Biggl\{#1\Biggr\}}

\def\norm#1{\Vert#1\Vert}

\def\abs#1{\vert#1\vert}
\def\babs#1{\bigl\vert#1\bigr\vert}
\def\bbabs#1{\Bigl\vert#1\Bigr\vert}
\def\bbbabs#1{\biggl\vert#1\biggr\vert}
\def\bbbbabs#1{\Biggl\vert#1\Biggr\vert}

\def\bmid{\big\vert}

\def\bbbmid{\bigg\vert}

\def\floor#1{{\lfloor#1\rfloor}}

\def\ceil#1{{\lceil#1\rceil}}

\def\convprob{\stackrel{P}{\longrightarrow}}

\newcount\minute
\newcount\hour
\newcount\hourMins
\def\now{%
\minute=\time%
\hour=\time \divide \hour by 60%
\hourMins=\hour \multiply\hourMins by 60%
\advance\minute by -\hourMins%
\zeroPadTwo{\the\hour}:\zeroPadTwo{\the\minute}%
}
\def\zeroPadTwo#1{\ifnum #1<10 0\fi#1}

\renewcommand\section{\@startsection {section}{1}{\z@}%
{-3.5ex \@plus -1ex \@minus -.2ex}%
{1.3ex \@plus.2ex}%
{\center\small\sc\mathversion{bold}\MakeUppercase}}

\def\subsection#1{\@startsection {subsection}{2}{0pt}%
{-3.5ex \@plus -1ex \@minus -.2ex}%
{1ex \@plus.2ex}%
{\bf\mathversion{bold}}{#1}}

\def\subsubsection#1{\@startsection{subsubsection}{3}{0pt}%
{\medskipamount}%
{-10pt}%
{\normalsize\itshape}{\kern-2.2ex. #1.}}

\def\blfootnote{\xdef\@thefnmark{}\@footnotetext}

\makeatother

\def\s#1{^{(#1)}}

\def\rfb#1{^{[#1]}}
\def\rf#1{^{\overline{#1}}}
\def\urnvar(#1,#2,#3,#4){X_{ #1,#2}(#3, #4)}
\def\urnvarshort(#1,#2,#3,#4){X_{ #1}}
\def\urnlaw(#1,#2,#3,#4){\mathcal{P}
^#2({\textstyle{#3\atop #4}};#1)}
\def\burnlaw(#1,#2,#3,#4){\mathcal{P}
^#2\bklr{{\textstyle{#3\atop #4}};#1}}
\def\bburnlaw(#1,#2,#3,#4){{\mathcal{P}
^#2\bbklr{{\textstyle{#3\atop #4\strut}};#1}}}
\def\limitlaw(#1,#2,#3){\mathrm{UL}(#1,#2,#3)}
\def\P(#1,#2,#3){\mathcal{P}({\textstyle{#2\atop #3}};#1)}

\newcommand\supp{\mathcal{S}}

\def\ul{\mathrm{UL}}

\begin{document}

\title{\sc\bf\large\MakeUppercase{P\'olya urns with immigration at random times}}
\author{\sc Erol Pek\"oz, Adrian R\"ollin and Nathan Ross}
\date{\it Boston University, National University of Singapore, University of Melbourne }
\maketitle

\begin{abstract}  
We study the number of white balls in a classical P\'olya urn model with the additional feature that, at random times, a black ball is added to the urn. The number of draws between these random times are i.i.d.\  and, under certain moment conditions on the inter-arrival distribution, we characterize the limiting distribution of the (properly scaled) number of white balls as the number of draws goes to infinity. The possible limiting distributions obtained in this way vary considerably depending on the inter-arrival distribution and are difficult to describe explicitly. However, we show that the limits are fixed points of certain probabilistic distributional transformations, and this fact provides a proof of convergence and leads to properties of the limits. The model can alternatively be viewed as a preferential attachment random graph model where added vertices initially have a random number of edges, and from this perspective, our results describe the limit of the degree of a fixed vertex.
\end{abstract}


\section{Introduction and main results}

P\'olya urn schemes form a rich class of fundamental probability models with a long history going back to \cite{Eggenberger1923} and extending to present day research. 
The standard general model is a recursive Markov process which begins with ``balls" of different colors in an urn, and at each step a ball is drawn randomly from the urn and returned along with the addition or removal of some prescribed number of balls of each color. The popularity of these models is due to the fact that variations of this basic P\'olya urn reinforcement mechanism appear in applications in biology, computer science, statistics, and elsewhere; see \cite{Pemantle2007} and
\cite{Mahmoud2009}.
Here we study the limiting behavior of a new urn model that is a simple variation of the classical P\'olya urn and which  arises naturally from a certain random graph model. Outside of application, the model is intrinsically interesting since the limiting behavior is subtle and intricately related to our method of proof, and other more standard techniques for analyzing urn models do not naturally apply (a more thorough discussion of existing literature and these other methods of proof can be found in Section~\ref{sec:lit}).
We now define our model and then state our main results.

Let~$\tau_1,\tau_2,\ldots$ be i.i.d.\ non-negative integer valued  random variables 
having  
distribution~$\pi=(\pi_k)_{k\geq0}$,
where we assume throughout that~$\pi_0<1$, and let~$T_j=\sum_{i=1}^j \tau_i$. 
It is helpful to think of the~$\tau_i$ as \emph{inter-arrival} interval lengths in a renewal process,
so that the~$T_j$ are the arrival times.
 Consider the following P\'olya urn model.
Initially, there are~$b$ black balls and~$w$ white balls. At each
step, a ball is drawn and replaced along with another of the same color.
Additionally, after draws
$T_1,T_2, \ldots$, regardless of the outcome of the draw,
a single extra black ball is added to the urn. Note that if~$\tau_i=0$, so~$T_i=T_{i+1}$, then more than one
black ball can be added to the urn between draws.

For example,
if
$(\tau_1,\ldots, \tau_5) = (1,3,0,0,4)$, then~$(T_1,\dots,T_5) = (1,4,4,4,8)$.
At Step~$1$, a regular P\'olya urn step is performed (that is, a ball is drawn and replaced along with a ball of the same color), 
and then one additional black ball is added since~$T_1 = 1$. At Steps~$2$ and~$3$, regular P\'olya urn steps are performed with no added black ball. Then, 
at Step~$4$, a regular P\'olya urn step is performed and then three additional black balls are added, since~$T_2=T_3=T_4=4$. Then, four regular P\'olya urn steps are performed after which another black ball is added ($T_5=8$). 
Note that given a black ball is added at a particular time step, 
the total number of balls added at that step has a geometric distribution 
(support starting at~$1$) with success probability~$1-\pi_0$. Note also that the number of black balls after Step 0 is not necessarily~$b$; this happens if~$\tau_1=0$, in which case black balls are added already before the first draw and replacement step is performed.

We study the  distribution of the number of white balls in the urn after~$n$ steps in this model, denoted  by~$\urnlaw(n,\pi,b,w)$. 
In particular we show that~$\urnlaw(n,\pi,b,w)$ properly scaled
converges in distribution to a non-standard limit law. 
The limits for deterministic~$\pi$ are studied in \cite{Janson2006} for~$\pi_1=1$ and \cite{Pekoz2014} for~$\pi_k=1$ with~$k>1$.
Before stating the result, we need to describe the limit.

\subsection{Urn limit laws}

Let~$v>0$, and let~$a_1,a_2,\dots$ be a sequence of non-negative numbers so that~$a_k>0$ for at least one~$k\geq 1$. Let 
\ben{\label{1}
	A(x) = \sum_{k\geq 1}a_k x^k,
}
and assume the radius of convergence~$\rho=\sup\{x\geq 0:A(x)<\infty\}$ is either positive or infinite. For such~$(a_k)_{k\geq1}$, we define the probability density 
\ben{\label{2}
	u(x) = c\,x^{v-1}\exp\bbklg{-v\int_0^x \frac{A(t)}{t}dt}, \qquad 0<x<\rho,
}
where~$c$ is an appropriate normalising constant depending on~$v$ and~$(a_k)_{k\geq 1}$, and we denote the corresponding probability distribution by~$\ul\bklr{v;(a_k)_{k\geq 1}}$. 
Specific instances of the~$\ul$ family include many standard non-negative continuous distributions
such as the exponential, Rayleigh, absolute normal, gamma, beta, and roots of gamma variables. 
We first establish that~\eq{2} is indeed a proper probability density and derive some basic properties of the laws~$\ul$.
\begin{lemma} Under the assumptions and notation above, the function~$u(x)$ defined by~\eq{2} is a probability density for an appropriate normalisation~$c$. Moreover, $Z\sim\ul\bklr{v;(a_k)_{k\geq 1}}$ has finite moments~$\mu_k = \IE Z^k$ of all orders, which satisfy the relation 
\ben{\label{3}
	\mu_k = \frac{v}{v+k}\sum_{l\geq 1} a_l \mu_{k+l}, \qquad \text{for all~$k\geq 0$.}
}
Furthermore, for~$\theta>0$,
\ben{
	\theta Z \sim \ul\bklr{v;(\theta^{-k} a_k)_{k\geq 1}}. \label{3a}
}
\end{lemma}

\begin{proof} For the first assertion, we show that the density given at~\eq{2}
has finite integral over~$(0,\rho)$. Let~$k_0\geq 1$ be such that~$a_{k_0}>0$. Observe that
\be{
	\Phi(x) := \int_{0}^x \frac{A(t)}{t}dt = \sum_{k\geq 1}\frac{a_k}{k}x^k
}	
so that
\be{
	x^{v-1}\exp\bklg{-v\Phi(x)}\leq x^{v-1}\exp\bklg{-va_{k_0}x^{k_0}/k_0},
}
which clearly has finite integral. Replacing~$x^{v-1}$ by any arbitrary power, we also conclude that all moments are finite. 
After noting that since the coefficients of~\eq{1}  are all non-negative, \cite[7.21]{Titchmarsh1958} implies that~$\lim_{x\to\rho^-}A(x)=\infty$ and hence also~$\lim_{x\to\rho^-}\Phi(x)=\infty$, the relation \eq{3} is just integration by parts; we have
\bes{
	\mu_k 
	&= c\int_0^\rho x^{k+v-1} \exp\bklg{-v\Phi(x)}dx 
	= c\int_0^\rho \frac{x^{k+v}}{k+v} \cdot v\frac{A(x)}{x}  \exp\bklg{-v\Phi(x)}dx \\
	& = \frac{v}{k+v} c\int_0^\rho x^{k+v-1} A(x) \exp\bklg{-v\Phi(x)}dx 
	 = \frac{v}{k+v} \sum_{l\geq 1}a_l \int_0^\rho x^{k+l}  u(x)dx.
}
Interchange of summation and integration is justified by the monotone convergence theorem, since all coefficients are non-negative. The final assertion~\eq{3a} is straightforward.
\end{proof}

\subsection{Limit results for urns with random immigration}

To state our first main result, let~$\Beta(\alpha, \beta)$, where~$\alpha$ and~$\beta$ are positive numbers, denote the law of 
the beta distribution supported on~$(0,1)$ with density proportional
to~$x^{\alpha-1}(1-x)^{\beta-1}$, and interpret~$\Beta(\alpha,0)$ as
the point mass at~$1$. We first consider the problem of convergence of moments of the (appropriately scaled) number of white balls in the urn. In what follows, we interpret~$\frac{\infty}{\infty+1}$ as~$1$.
Here and below, $C$ is a generic constant that may change from line to line.

\begin{theorem}\label{THM1}
Let~$b$ and~$w$ be positive integers, let~$\pi$ be a probability distribution on the non-negative integers with mean~$0<\mu\leq\infty$,
and let~$\tau\sim\pi$.
If~$k$ is an integer such that either
\begin{enumerate}
\item[$(a)$]~$\IE \tau^p <\infty$ for some~$p>1$, and~$1\leq k < (\frac{p}{2}-1)(\mu+1)-1$, or
\item[$(b)$] there is~$\eps>0$ such that~$\IP[\tau > n] \geq Cn^{-(1-\eps)}$ for~$n$ large enough, and~$k\geq 1$, 
\end{enumerate}
then there is a positive constant~$m_k(b,w,\pi)$ such that, for~$X_n\sim\urnlaw(n,\pi,b,w)$, we have
\ben{\label{7}
\IE\bbbklg{\bbklr{\frac{X_n}{n^{\mu/(\mu+1)}}}^k} \to  m_k(b,w,\pi)\qquad\text{as~$n\to\infty$}.
}
\end{theorem}

We now formulate the main distributional convergence result
which essentially says that when~$b=1$, the scaled urn limits are 
of the form~$\ul\bklr{w;(a_k)_{k\geq 1}}$ for appropriate choice of~$(a_k)_{k\geq 1}$, and when~$b>1$, the limits are in the same family 
up to multiplication by an independent beta variable. 

\begin{theorem}\label{THM2} 
Let~$b$ and~$w$ be positive integers, and let~$\pi$ be a probability distribution on the non-negative integers with mean~$0<\mu\leq \infty$, and let~$\tau\sim\pi$. Assume that either
\begin{enumerate}
\item[$(a)$]~$\IE \tau^p <\infty$ for all~$p\geq 1$, or
\item[$(b)$] there is~$\eps>0$ such that~$\IP[\tau > n] \geq Cn^{-(1-\eps)}$ for~$n$ large enough, 
\end{enumerate}
and let~$m_k(1,b+w-1,\pi)$, $k\geq 1$, be as in Theorem~\ref{THM1}. Set
\ben{\label{8}
	a_k = \frac{\pi_{k-1}}{m_k(1,b+w-1,\pi)}, \qquad \text{for all~$k\geq1$,}
}
and let~$Z\sim\ul\bklr{b+w-1;(a_k)_{k\geq 1}}$. Then 
\ben{\label{9}
 \IE Z^k=m_k(1,b+w-1,\pi),\qquad \text{for all~$k\geq1$,}
}
and, with~$X_n\sim\urnlaw(n,\pi,b,w)$,
\ben{\label{10}
	\law\bbbklr{\frac{X_n}{n^{\mu/(\mu+1)}}}\to\law(BZ)\qquad \text{as~$n\toinf$,}
}
where~$B\sim\Beta(w,b-1)$ is independent of~$Z$.
\end{theorem}

Our expressions for the moments of~\eq{9} are not explicit and so then neither are the parameters of the limits, which leads to many intriguing open questions; see Section~\ref{sec2}. In the case where~$\pi$ is deterministic, the limiting distributions can be described explicitly in a number of ways, see \cite{Janson2006}, \cite{Pekoz2013a,Pekoz2014}, and so it is interesting that adding randomness in this way leads to limiting distributions that are complicated and difficult to describe.

The moment results of Theorem~\ref{THM1} follow by first deriving formulas
conditional on 
the partial sums~$(T_1,T_2,\ldots)$ of the i.i.d.\ inter-arrival times~$\tau_1,\tau_2,\ldots$,
and then using classical moment and concentration inequalities for such quantities.
The moment 
results show that the sequence~$n^{-\mu/(\mu+1)}X_n$, $n\geq 1$, is 
tight as long as we have either~$\mu<\infty$ and~$\IE \tau^{6/(\mu+1)+2}<\infty$ 
or~$\IP[\tau > n] \geq Cn^{-(1-\eps)}$, and so in these cases,
a distributional limit follows by showing uniqueness of subsequential limits.
For the case~$b=1$, we are able to show that in the two cases just described, any subsequential limit 
is a  fixed point (unique given moments) of a certain distributional transformation
which we describe in Section~\ref{sec2a} below.

The organization of the remainder of the paper is as follows. We finish this section 
with a discussion of related literature and then provide a connection between our model and preferential attachment graphs
with random number of initial attachments for each vertex. In 
Section~\ref{sec2a} we describe the distributional fixed point equation 
used to identify the limits appearing in Theorem~\ref{THM2}.
Our study leads to many further questions, especially around descriptions
of the limits and moment sequences appearing in Theorem~\ref{THM1} and~\ref{THM2}, and so we discuss some of these in Section~\ref{sec2}, where we also list open problems and conjectures.
Section~\ref{sec3}
contains the proof of Theorem~\ref{THM1}, Section~\ref{sec4} has the proof of Theorem~\ref{THM2}, and in Section~\ref{sec:props} we derive some basic properties of the~$\ul$ family.

\subsection{Related Literature}\label{sec:lit}

The literature around P\'olya urn models is too vast for a complete survey, but the the main results and modern techniques are well covered by  \cite{Chauvin2011}, \cite{Chauvin2015}, \cite{Chen2005}, \cite{Chen2013}, \cite{Flajolet2005}, \cite{Janson2004, Janson2006}, \cite{Knape2014}, \cite{Kuba2015a, Kuba2015}, \cite{Laruelle2013}, \cite{Pouyanne2008}, and references therein.  
These papers cover many variations of the standard model, including random replacement rules and drawing multiple balls at a time. Techniques used to study limits include finding appropriate martingales, stochastic approximation, embedding the process into continuous time branching processes, deriving moments or moment generating functions using analytic or algebraic relations derived from the Markovian dynamics of the process, and the contraction method. 
All of these methods rely on a reasonably nice Markovian dynamics and in general, the model studied here is not Markov in its natural time scale. It is possible to make the model Markov by observing the process at the random times of immigration, but then the dynamics are complicated, and so it is challenging to apply the techniques mentioned above.
On the other hand, our distributional fixed point approach is naturally suited to the model and 
 leads to intriguing descriptions of the limiting behavior and to further avenues of study. We leave 
 the question of what can be learned by studying this model with other methods to further work (see Section~\ref{sec2}).

\subsection{Connection to preferential attachment random graph}\label{sec1}
In preferential attachment random graph models, vertices are sequentially added 
and randomly connected to existing vertices 
such that connections to higher degree nodes are more likely.
There are many variations of these popular models;
a good reference is \cite[Chapter~8]{Hofstad2016}.

Consider the following sequence~$(G(n))_{n\geq0}$ of
preferential attachment random graphs.
The initial state~$G(0)$ is a ``seed" graph with~$s$ vertices,
where the degree or ``weight" of vertex~$1\leq i \leq s$ is~$d_i>0$. 
We denote the weight of vertex~$i$ in~$G(n)$ by~$d_i(n)$ so 
note for~$1\leq i\leq s$, $d_i(0)=d_i$.

Let~$\tau_1,\tau_2,\ldots$ be i.i.d.\ distributed according to inter-arrival distribution~$\pi$.
Given the graph~$G(n-1)$ having~$s+n-1$ vertices,
$G(n)$ is formed by adding 
a vertex labeled~$s+n$ and sequentially attaching~$\tau_n$
edges between it and the
vertices of~$G(n-1)$ 
according to the following rules. 
The first edge attaches
to vertex~$k$ with
probability
\ben{\label{11}
\frac{d_k(n-1)}{\sum_{i=1}^{s+n-1}d_i(n-1)}, \qquad 1\leq k\leq n-1;
}
denote by~$K_1$ the vertex which received that first edge. The weight of~$K_1$ is updated immediately, so that
the second edge attaches to vertex~$k$ with probability
\be{
\frac{d_{k}(n-1)+\I[k=K_1]}{1+\sum_{i=1}^{s+n-1}d_i(n-1)}, \qquad 1\leq k\leq n-1.
}
The procedure continues this way,
edges attach with probability proportional to weights at that moment, and additional
received edges add one to the weight of a vertex, until vertex~$n$ has~$\tau_n$ outgoing edges.
Lastly, we set~$d_{s+n}(n)=1$, and let~$G(n)$ be the resulting graph. Note that multiple
edges between vertices are possible.

This model is a randomized version of the ``sequential" model of \cite{Berger2014}; also
the ``N$_\ell$" model of \cite{Pekoz2014a}.
For related models where the 
number of edges are random but the updating rule is not sequential (meaning
each of the~$\tau_n$ edges of vertex~$s+n$ attach with probability~\eq{11})
see \cite{Deijfen2009} and a particular choice of parameters in the general model of \cite{Cooper2003}.

Writing~$c_i:=\sum_{j=1}^i d_j$, 
the connection between the preferential attachment model above and our urn model
 is that for~$1\leq k < s$,
\be{
\law\bbbklr{\sum_{i=1}^k d_i(n)}=\bburnlaw(\sum_{i=1}^{n} \tau_i, \pi, c_s-c_k, c_k),
}
where the~$\tau_i$'s on the right hand side drive the urn process. 
Thus for~$\pi$ having all positive integer moments finite, 
we have (in particular)~$T_n/n:=n^{-1}\sum_{i=1}^n \tau_i \to \mu$ almost surely and so
Theorem~\ref{THM1} implies that for~$k=1,\ldots, s$,
\be{
\law\bbbbklr{\frac{\sum_{i=1}^{k} d_i(n)}{(\mu n)^{\mu/(\mu+1)}}}\to \law(B Z), 
}
where, in accord with Theorem~\ref{THM2}, $Z\sim\ul\bklr{c_s-1,\bklr{\tsfrac{\pi_{k-1}}{m_k}}_{k\geq 1}}$,
$m_k=m_k(1,c_s-1,\pi)$ are the limiting moments given in~\eq{7} of Theorem~\ref{THM1}, and~$B\sim\Beta(c_k,c_s-c_k-1)$ is independent of~$Z$.

For later vertices, if~$k\geq s$, then for~$n\geq k-s+1$,
\ben{\label{12}
\law\bbbklr{\sum_{i=1}^k d_i(n)}=\bburnlaw(\sum_{i=k-s+2}^{n} \tau_i, \pi, 1, {c_s+(k-s)+ T_{k-s+1}}),
}
where again the~$\tau_i$'s on the right hand side drive the urn process. 
Given~$\tau_1, \ldots, \tau_{k-s+1}$, it is still the case that~$n^{-1}\sum_{i=k-s+2}^n \tau_i \to \mu$
and thus 
\be{
\law\bbbbklr{\frac{\sum_{i=1}^{k} d_i(n)}{(\mu n)^{\mu/(\mu+1)}}\Bigg| \bklr{\tau_1,\ldots, \tau_{k-s+1}} }
\to \ul\bklr{c_s+(k-s)+ T_{k-s+1}, \bklr{\tsfrac{\pi_{k-1}}{m_k}}_{k\geq1}},
}
where~$m_k=m_k(1,c_s+(k-s)+ T_{k-s+1},\pi)$ is the limiting moment 
sequence~\eq{7} of Theorem~\ref{THM1}. 
Thus the (unconditional) limiting cumulative degree counts are an appropriate mixture of the
$\ul$ laws.

To our knowledge, these are the first results regarding the degree of fixed vertices in 
preferential attachment models with random initial degrees. The degree of 
a randomly chosen node is studied in \cite{Deijfen2009} and \cite{Cooper2003}.

\section{Distributional fixed point equation}\label{sec2a}
 
To describe the distributional fixed point equation 
used to identify the limits appearing in Theorem~\ref{THM2}, we 
first need a preliminary distributional transformation. 

 \begin{definition}
Let~$\psi$   be a probability distribution concentrated on the non-negative integers, and let~$X$ be a positive random variable such that~$\IE X^{k}<\infty$ for all~$k$ for which~$\psi_k>0$.
A random variable~$X\s{\psi}$ is said to have the \emph{$\psi$-power-bias
distribution of~$X$} if
\ben{\label{5}
	\IE f\bklr{X\s{\psi}} = \sum_{k\,:\,\psi_k>0} \psi_k  \frac{\IE\bklg{ X^{k} f(X)}}{\IE X^{k}}
}
for all~$f$ for which the expectation on the right hand side exists.
 \end{definition}

If~$\psi_1=1$, then the~$\psi$-power-bias 
distribution is commonly known as the \emph{size-bias distribution}; 
see for example \cite{Arratia2013} and \cite{Brown2006}.
If~$\psi_k=1$ for some~$k\geq 2$, then the~$\psi$-power-bias 
distribution is sometimes referred to as the \emph{$k$-power bias distribution}, denoted by~$X\s k$. 
We can realize~$X\s{\psi}$ by first
sampling a random index~$K$ according to~$\psi$, 
and conditional on~$K=k$,  we let~$X\s{\psi}$ have the~$k$-power-bias distribution of~$X$.
This description implies that the~$\psi$-power-bias transformation 
may be amenable to analysis in our setting
since constructing constant~$k$-power bias
distributions is understood in P\'olya urn models 
\cite{Pekoz2013a,Pekoz2013, Pekoz2014}, \cite{Ross2013},
and other discrete probability applications \cite{Barbour1992},
\cite{Chen2011}, \cite{Bartroff2013}.

To establish
 the distributional transformation for which~$\ul\bklr{w;(a_k)_{k\geq1}}$ is a fixed point, note
 first that if~$Z\sim\ul\bklr{w;(a_k)_{k\geq1}}$ and~$\mu_k=\IE Z^k$, then 
\eq{3} yields, in particular,
\be{
	\sum_{k\geq 1} a_k \mu_{k} = 1,
}
so that
\ben{\label{4}
	\psi_k = a_{k} \mu_{k}, \qquad k\geq 1
}
defines a probability distribution on the positive integers, which, by~\eq{3a}, is invariant to scaling of~$Z$. 
The next result gives the~$\ul$ family as fixed point of a distributional transformation; the connection 
between this transformation and the representation~\eq{2} was first made in \cite{Pakes2007}.

\begin{proposition}\label{prop1} The following holds. 
\begin{itemize}\item[$(i)$]  If~$X\sim\ul\bklr{w;(a_k)_{k\geq1}}$ and~$\psi$ is defined as in \eq{4}, then
\ben{\label{6}
	\law(X) = \law\bklr{V_w X\s{\psi}},
}
where~$V_w\sim\Beta(w,1)$ is independent of~$X\s{\psi}$.
\item[$(ii)$] Let~$w>0$, and let~$\psi$ be a probability distribution on the positive integers. If~$X$ is a positive random variable such that~$\IE X^k<\infty$ whenever~$\psi_k>0$ and \eq{6} holds, then~$A(x)$, defined with respect to the sequence~$a_k = \psi_k/\IE X^k$ if~$\psi_k>0$ and~$a_k=0$ otherwise, has positive or infinite radius of convergence and~$X \sim \ul\bklr{w;(a_k)_{k\geq1}}$.
\end{itemize}
\end{proposition}

\begin{proof} To prove~$(i)$, assume~$X\sim\ul(w;(a_n)_{n\geq 1})$. Using the formula for the density of products of independent random variables
and denoting the density of~$X\s{\psi}$ by~$u\s{\psi}$, we obtain that~$V_w X\s{\psi}$ has density
\bes{
 	&\int_{x}^\rho w\bbbklr{\frac{x}{t}}^{w-1}\frac{u\s{\psi}(t)}{t} dt 
 	 = \int_{x}^\rho w\bbbklr{\frac{x}{t}}^{w-1}\frac{u(t)}{t}\sum_{n\geq 1} \frac{\psi_n}{\mu_n}t^ndt \\
 	&\qquad = \int_{x}^\rho w\bbbklr{\frac{x}{t}}^{w-1}\frac{u(t)}{t}\sum_{n\geq 1} a_{n}t^ndt
 	= \int_{x}^\rho w\bbbklr{\frac{x}{t}}^{w-1} u(t) \frac{A(t)}{t}dt \\
	&\qquad= x^{w-1}\int_{x}^\rho w\frac{A(t)}{t} \frac{u(t)}{t^{w-1}}dt=u(x).
} 

To prove~$(ii)$, assume~$X$ is a positive random variable such that
~$\IE X^k<\infty$ for all~$k\geq 1$ with~$\psi_{k}>0$, and assume~\eq{6} holds. Since, by Jensen's inequality,
\be{
	\sum_{k\geq 1}\frac{\psi_k}{\IE X^k} x^k\leq
	\sum_{k\geq 1}\frac{1}{(\IE X)^k} x^k <\infty
}
whenever~$x<\IE X~$, the radius of convergence of~$A(x)$ must be at least~$\IE X$, which is positive since~$X$ is positive.
It follows from \eq{6} 
that~$X$ has a density, and the representation~\eq{2}
then follows from  \cite[Theorem 3.1]{Pakes2007}. 
\end{proof}

\begin{remark}\label{rem2} It is important to note that Proposition~\ref{prop1} does not answer the question whether, for given~$w>0$ and probability distribution~$\psi$, there is an~$X$ satisfying \eq{6}. It merely says that, if such~$X$ exists, then it has to be from the family~$\ul(w;(a_k)_{k\geq1})$, where~$(a_k)_{k\geq1}$ can be expressed in terms of~$\psi$ and the moments of~$X$. Note also that, for given~$w$ and~$\psi$, there might a priori be more than one~$(a_k)_{k\geq1}$ satisfying~$\psi_k = a_k\IE X^k$; see the discussion in the next section.
\end{remark}

Proposition~\ref{prop1} suggests that if a random variable~$W$ is such that~$\law(W)$ is close in an appropriate sense to~$\law(V_w W\s{\psi})$, then~$\law(W)$ is close to~$\ul(w,(\psi_k/\IE W^k)_{k\geq1})$. We formalize this as a convergence statement in Lemma~\ref{lem7} in Section~\ref{sec4}. We then apply this result to our urn models, where~$\psi$ has the immigration distribution~$\pi$, but shifted by one, and the limiting moments are those given by Theorem~\ref{THM1}. That the urn law and its transformation are close is achieved by coupling, in particular that power-biasing our urn models corresponds to adding extra white balls before starting the process (Lemma~\ref{lem9}), and that multiplying by a beta corresponds to running a classical P\'olya urn (Lemma~\ref{lem13}); see Section~\ref{sec4} for details.

\section{Open problems}\label{sec2}

We discuss some of the many questions that are not answered by our study.

\begin{question}\label{qn1} Are solutions to the distributional fixed point equation \eq{6} unique up to scaling? 
\end{question}

This is the most pressing open problem, and a positive answer would have a large impact on our understanding of the relation between limits of our urn model and the family of distributions~$\ul\bklr{w;(a_k)_{k\geq 1}}$. The main consequence of a positive answer would be the following ``inversion'' of Theorem~\ref{THM2}.

\begin{conjecture}\label{con1}
Fix~$w$ and~$(a_k)_{k\geq 1}$, and let~$Z\sim\ul\bklr{w;(a_k)_{k\geq1}}$. Then, with~$\pi_k=a_{k+1}\IE  Z^{k +1}$ for~$k\geq 0$, (and possibly further conditions on~$(a_k)_{k\geq1}$), 
the sequence~$X_n\sim\urnlaw(n,\pi,1,w)$ satisfies
\be{
	\law\bbbklr{\frac{X_n}{n^{\mu/(\mu+1)}}} \to \law(\theta Z),
}
where~$\theta = m_1(1,w,\pi)/\IE Z$ with~$m_1(1,w,\pi)$ given by Theorem~\ref{THM1}.
\end{conjecture}

It is clear that, for any~$(a_k)_{k\geq1}$ with positive or infinite radius of convergence, we can define the probability distribution~$\pi_k=a_{k+1}\IE  Z^{k +1}$, $k\geq 0$, and consider the limit of the corresponding urn model. But unless the solution to~$\eq{6}$ is unique up to scaling, our method of proof does not guarantee that the corresponding urn limit is a scaling of the one given by this~$(a_k)_{k\geq1 }$.

The following question recasts Question~\ref{qn1} differently; it must have a positive answer if Question~\ref{qn1} has a negative answer. 

\begin{question}\label{qn2} Fix~$w>0$. Are there two sequences~$(a_k)_{k\geq 1}$ and~$(\tilde a_k)_{k\geq 1}$ such that~$Z\sim\ul\bklr{w;(a_k)_{k\geq1}}$ and~$\tilde Z\sim\ul\bklr{w;(\tilde a_k)_{k\geq1}}$ are not scaled versions of each other, but such that~$a_k\IE Z^k = \tilde a_k \IE \tilde Z^k$ for all~$k\geq 1$? \end{question}

If Question~\ref{qn2} could be answered positively, then we would have a counter example to Conjecture~\ref{con1} --- both~$(a_k)_{k\geq 1}$ and~$(\tilde a_k)_{k\geq 1}$ would give rise to the same immigration distribution, but the corresponding urn model could converge to at most one of them.

One issue with Theorem~\ref{THM1} is 
that the limiting moments~$m_k$ are defined rather indirectly, and they are are difficult to calculate explicitly; the same comment applies
to moments of the~$\ul$ family.

\begin{question} Are there a more explicit formulas for~$m_k$ in
Theorem~\ref{THM1} in terms of~$w$ and~$\pi$; or for the moments of~$\ul\bklr{w;(a_k)_{k\geq1}}$ in terms of~$w$ and~$(a_k)_{k\geq 1}$?
\end{question}
%
%
There are a few examples where we can make explicit calculations and partially address this last question; see the end of this section.
%
%
%
%

A natural example that we have struggled to prove anything more specific 
 about than the conclusion of Theorem~\ref{THM2}, is for~$\pi$ a positive geometric variable. In this case the urn model can be described as follows: at each step, a P\'olya urn step is performed and then a~$p$-coin is tossed to determine if an additional black ball is added to the urn. So the process is Markovian, which could make more detailed analyses possible. 

\begin{question} 
What is a concrete description of the distributional limit of~$\urnlaw(n,\pi,1,w)$ (properly scaled) when~$\pi$ is a positive geometric distribution (support starting at~$1$)? 
\end{question}

%
%
%
\begin{question} There are a large number of ways the model can be generalized: more colors, different replacement rules. What can be said in these cases?
\end{question}
\begin{question} Can other methods, such as those described in Section~\ref{sec:lit}, be applied to strengthen our results? For example, if appropriate martingales can be found, then the convergence can be strengthened to almost sure and in~$L_p$ for appropriate~$p$.
\end{question}
We conclude this section with three examples.

\subsection{Explicit choices of~$\pi$}\label{sec:2a}

The relationship between the sequences~$(a_k)_{k\geq1}$ and~$\pi$ appearing in Theorem~\ref{THM2}
is rather implicit, and so in this section, we work out some examples where explicit calculations are possible.

\begin{example}[Deterministic~$\pi$]
If~$\pi_k=1$ for some~$k\geq1$, then the 
scaled limit of~$\urnlaw(n,\pi,1,w)$ has density proportional to
\be{
x^{w-1} \exp\{-w x^{k+1}/((k+1) m_{k+1}) \} dx,
}
which is the same as an appropriately scaled, standard gamma variable with parameter~$w/(k+1)$,
raised to the power~$1/(k+1)$.
For~$k=1$, the urn model is a time homogeneous triangular urn and
the limit can be read from \cite{Janson2006}. 
The general case is studied in detail in~\cite{Pekoz2014},
where rates of convergence to the limit are also provided. 
The limiting moments can be made explicit as well as the
constant~$m_{k+1}$.
\end{example}

\begin{example}[Bernoulli inter-arrival distribution]
We study~$\urnlaw(n,\pi,1,w)$ where~$\pi_0=1-\pi_1\not=1$. 
Note that for this choice of~$\pi$, at each step a P\'olya urn step is performed
and then a geometric with parameter~$\pi_1$ (support started at~$0$) distributed number of black balls
are added to the urn.
In the spirit of
Conjecture~\ref{con1}, we start with a positive integer~$w$ and positive numbers~$a_1$ and~$a_2$,
and then use these to determine~$\pi$.

First, define the function~$U$ for 
$a>0$, $z>0$ and~$b\in\IR$ by
\be{ 
U(a, b, z)=\frac{1}{\Gamma(a)} \int_0^\infty e^{-zt} t^{a-1} (1+t)^{b-a-1} dt.
}
This function is known as Kummer U (also called the confluent hypergeometric function of the second kind; see
\cite[13.2.5]{Abramowitz1964}).
Second, we calculate the normalising constant~$c$ in~\eq{2}; one can show that
\be{
	\int_0^\infty x^{w-1} 
	\exp\left\{-w\left(a_1 x+ a_2 x^2/2\right)\right\} dx=
	\frac{\Gamma(w) U\bklr{\tsfrac{w}{2}, \tsfrac{1}{2}, 	
	\tsfrac{a_1^2 w }{2 a_2}}}{(2a_2 w)^{w/2}},
}
so that
\ben{\label{13}
u(x) = \frac{(2a_2 w)^{w/2}}{\Gamma(w) U\bklr{\tsfrac{w}{2}, \tsfrac{1}{2}, \tsfrac{a_1^2 w }{2 a_2}}}
 x^{w-1} 
\exp\bklg{-w\left(a_1 x+ a_2 x^2/2\right)} \qquad\text{ for~$x>0$.}
}
Third, we calculate the relevant moments and obtain
\ben{\label{13a}
\IE  Z=\frac{w a_1 }{2a_2}\cdot \frac{U\bklr{\frac{w}{2}+1,\frac{3}{2},\frac{w a_1^2 }{2a_2}}}{U\bklr{\frac{w}{2},\frac{1}{2},\frac{w a_1^2  }{2a_2}}},
\qquad
\IE  Z^2=\frac{1 + w}{2 a_2}\cdot\frac{U\bklr{\frac{w}{2}+1, \frac{1}{2}, \frac{w a_1^2 }{2a_2}}}
{U\bklr{\frac{w}{2}, \frac{1}{2},  \frac{w a_1^2 }{2a_2}}},
}
Putting this together we obtain
\ben{\label{14} 
\pi_0=1-\pi_1=\frac{w a_1^2  }{2a_2}\cdot \frac{U\bklr{\frac{w}{2}+1,\frac{3}{2},\frac{w a_1^2 }{2a_2}}}{U\bklr{\frac{w}{2},\frac{1}{2},\frac{w a_1^2  }{2a_2}}}
=\frac{U\bklr{\frac{w+1}{2},\frac{1}{2},\frac{w a_1^2 }{2a_2}}}{U\bklr{\frac{w+1}{2},\frac{3}{2},\frac{w a_1^2  }{2a_2}}}.
}
The second equality of~\eq{14} follows by applying the identity~$U(a,b,z)=z^{1-b} U(1+a-b, 2-b, z)$ (see \cite[13.1.29]{Abramowitz1964})
to both the numerator and the denominator of the middle expression of~\eq{14} with~$z=wa_1^2/(2a_2)$
and~$a=w/2+1$ and~$b=3/2$, respectively, $a=w/2$ and~$b=1/2$. As a check on~\eq{13a}, we can see directly that 
\be{
\pi_0+\pi_1=a_1 \IE  Z+a_2 \IE  Z^2=1
}
from \cite[13.4.18]{Abramowitz1964} with~$a=w/2+1$, $b=1/2$, and~$z=(wa_1^2)/(2a_2)$.

Theorems~\ref{THM1} and~\ref{THM2} give moment and distributional convergence
results for the urn model with inter-arrival distribution~$(\pi_0,\pi_1)$. Furthermore,
it is possible to show directly that for fixed~$w$, the function on positive pairs of numbers
\be{
	(a_1, a_2) \mapsto \pi_0
}
is surjective on~$(0,1)$; hence, every inter-arrival distribution concentrated on~$\{0,1\}$ can be generated
by starting with an appropriate~$a_1$ and~$a_2$.
Finally, we note that Conjecture~\ref{con1} is verified in this case since if 
\be{
\tilde a_1 \IE  \tilde Z=a_1 \IE Z,
\qquad
\tilde a_2 \IE  \tilde Z^2=a_2 \IE Z^2,
}
then~\eq{13a} implies that~$\tilde a_1^2/\tilde a_2=a_1^2/a_2$, which implies
$\tilde a_1^2/a_1^2 = \tilde a_2/a_2=:\theta^2$ (which is the same as the conjecture, noting~\eq{3a}).
\end{example}

\begin{conjexample}[Power law inter-arrival distribution]
Let~$\alpha$ and~$\beta$ be positive numbers, and set~$a=(\beta\alpha^{-1},\beta\alpha^{-2}, \beta\alpha^{-3}, \ldots)$. Then for~$0<x<\alpha$,
\be{
\sum_{k\geq 1} a_k x^k/k=-\beta \log(1- x/\alpha).
}
Thus, if~$Z\sim\ul(w; (a_k)_{k\geq1})$ has density given by~\eq{2} with~$v=w$, we find~$\law( \alpha^{-1} Z)=\Beta(w, w\beta+1)$
and that
\be{
\IE  Z^j=\alpha^{j} \frac{ \Gamma(w(\beta+1) +1)\Gamma(w+j)}{\Gamma(w)\Gamma(w(\beta+1) +j+1)}.
}
Following the blueprint of Conjecture~\ref{con1}, define for~$j=0,1,\ldots,$
\ben{\label{42}
\pi_j=a_{j+1}\IE  Z^{j+1}=\beta \frac{ \Gamma(w(\beta+1) +1)\Gamma(w+j+1)}{\Gamma(w)\Gamma(w(\beta+1) +j+2)}.
}
These calculations suggest that if~$\alpha$, $\beta$ and~$w$ are positive numbers, and~$\pi$ has distribution
given by~\eq{42},
then 
there is a constant~$\theta>0$ such that,
for~$X_n \sim \urnlaw(n,\pi,1,w)$ and as~$n\to\infty$,
\be{
	\law\bklr{ n^{-\frac{\mu}{\mu+1}}X_n} \to
	 \law(\theta  Z),
}
where~$\law(\alpha^{-1}   Z)= \Beta(w,w\beta+1)$
and~$\mu$ denotes the mean of~$\pi$ given by
\be{
\mu=
\begin{cases}
	(w+1)\left(\beta w-1\right)^{-1} & \text{if~$w\beta >1$,} \\
\infty & \text{if~$w\beta \leq 1$.}
\end{cases}
}

The previous statement is conjectural for two reasons. In the case that~$w\beta>1$, $\pi$ has finite mean but not all moments finite, so even convergence in this case is not covered by Theorem~\ref{THM2}. For~$w\beta<1$, Theorem~\ref{THM2} applies and says that 
$\law\bklr{ n^{-\frac{\mu}{\mu+1}}X_n}$ converges in distribution to
$\ul(w, (\pi_k/m_k(1,w,\pi))_{k\geq1})$, but without a result like Conjecture~\ref{con1},
we cannot conclude that~$\pi_k/m_k(1,w,\pi)=\theta^k \beta \alpha^{-k}$ for some~$\theta>0$.

Also note that~$\beta\to0$ roughly corresponds to~$\pi_k=0$ for all~$k$,
and so~$\tau=\infty$,
which should behave as a classical P\'olya urn, and indeed 
the conjectured limit tends to the anticipated~$\Beta(w,1)$.
In general the~$\pi$ distribution in this case is heavy-tailed and the conjecture
suggests that extra balls are not added with enough frequency
to get too far away from the classical P\'olya urn.
\end{conjexample}

\section{Proof of Theorem~\ref{THM1}}\label{sec3}

We will show the following result, which is the analogue of Theorem~\ref{THM1}, but for factorial moments. In what follows, we interpret products~$\prod_{j=a}^b$ as~$1$ whenever~$b<a$.

\begin{proposition}\label{thm1}
Let~$b$ and~$w$ be positive integers, let~$\pi$ be a probability distribution on the non-negative integers with mean~$0<\mu\leq \infty$, and let~$\tau\sim\pi$. Let~$X_n\sim\urnlaw(n,\pi,b,w)$, and set 
\be{
	D_{k,n}=\prod_{j=0}^{k-1} \left(X_n+j\right),
	\qquad \text{$k\geq 1$, $n\geq 0$.}
}
If either
\begin{enumerate}
\item[$(a)$]~$\IE \tau^p <\infty$ for some~$p>1$, and~$1\leq k < (\frac{p}{2}-1)(\mu+1)-1$ , or
\item[$(b)$] there is~$\eps>0$ such that~$\IP[\tau > n] \geq Cn^{-(1-\eps)}$ for~$n$ large enough, and~$k\geq 1$, 
\end{enumerate}
then there is a positive constant~$m_k(b,w,\pi)$ such that
\ben{\label{15}
\IE\bbbklg{\frac{D_{k,n}}{n^{k\mu/(\mu+1)}}} \to  m_k(b,w,\pi)\qquad\text{as~$n\to\infty$}.
}
\end{proposition}

To prove the proposition we need some lemmas. We first establish a moment formula for~$D_{k,n}$, conditional on the immigration times~$T_1,T_2,\dots$.

\begin{lemma}\label{lem2} Let~$D_{k,n}$ be as in Proposition~\ref{thm1}, let~$T=(T_1, T_2,\ldots)$ be the sequence of immigration times of the process, and for~$n\geq 0$, let~$N_n=\#\{i\geq 1: T_i \leq n\}$, the number of immigrations up to and including draw~$n$.
Then, for any~$k\geq 1$ and~$n\geq 1$,
\ban{
	&\IE( D_{k,n} | T )
	=\frac{\Gamma(w+k)}{\Gamma(w)}  \prod_{j=0}^{n-1}\frac{b+w+k+j+N_j}{b+w+j+N_j} \label{16} \\ 
	&\qquad= \frac{\Gamma(w+k)\Gamma(b+w)}{\Gamma(w)\Gamma(b+w+k)}
\frac{\Gamma(b+w+N_{n-1}+n+k)}{\Gamma(b+w+N_{n-1}+n)}
\prod_{j=1}^{N_{n-1}}\frac{b+w+j-1+T_j}{b+w+j-1+T_j+k}.\label{17}
}
\end{lemma}
\begin{proof} Let~$X_n \sim\urnlaw(n,\pi,b,w)$. To shorten the formulas, let~$c = b+w$.
Since the total number of balls in the urn after draw~$n-1$ is~$c+N_{n-1}+n-1$, we have 
\be{
\IP\bkle{X_n=X_{n-1}+1\,\bmid\, T,X_{n-1}}= \frac{X_{n-1}}{c+N_{n-1}+n-1},
}
and we easily find
\be{
\IE(D_{k,n}\,|\, T,X_{n-1})= D_{k,n-1} \frac{c+k+N_{n-1}+n-1}{c+N_{n-1}+n-1}.
}
Iterating yields 
\be{
\IE(D_{k,n}| T)= \IE(D_{k,0}|T)\prod_{j=0}^{n-1}\frac{c+k+j+N_j}{c+j+N_j},
}
which is easily seen to be~\eq{16}. 
Now, set~$T_0=0$ and note that for~$i\geq 1$, if~$T_{i-1}<T_i$, then~$N_{T_{i-1}}=\cdots=N_{T_{i}-1}=i-1$,  so we can rewrite this last expression as 
\bes{
\frac{\IE( D_{k,n}| T)}{\IE(D_{k,0}|T)} 
&=\left(\prod_{i=1}^{N_{n-1}}\prod_{j=T_{i-1}}^{T_{i}-1}\frac{c+k+j+i-1}{c+j+i-1} \right)\prod_{j=T_{N_{n-1}}}^{n-1}\frac{c+k+j+N_{n-1}}{c+j+N_{n-1}} \\
&=\left(\prod_{i=1}^{N_{n-1}}\frac{\Gamma(c+k+i-1+T_i)\Gamma(c+i-1+T_{i-1})}{\Gamma(c+k+i-1+T_{i-1})\Gamma(c+i-1+T_i)} \right )\\
&\kern13em\times \frac{\Gamma(c+k+N_{n-1}+n)\Gamma(c+N_{n-1}+T_{N_{n-1}})}{\Gamma(c+k+N_{n-1}+T_{N_{n-1}})\Gamma(c+N_{n-1}+n)} \\
&=\left(\prod_{i=1}^{N_{n-1}}\frac{\Gamma(c+k+i-1+T_i)}{\Gamma(c+i-1+T_i)}\right)
\left(\prod_{i=0}^{N_{n-1}-1}\frac{\Gamma(c+i+T_{i})}{\Gamma(c+k+i+T_{i})} \right )\\
&\kern13em\times \frac{\Gamma(c+k+N_{n-1}+n)\Gamma(c+N_{n-1}+T_{N_{n-1}})}{\Gamma(c+k+N_{n-1}+T_{N_{n-1}})\Gamma(c+N_{n-1}+n)} \\
&=\left(\prod_{i=1}^{N_{n-1}}\frac{\Gamma(c+k+i-1+T_i)\Gamma(c+i+T_{i})}{\Gamma(c+k+i+T_{i})\Gamma(c+i-1+T_i)} \right )\\
&\kern13em\times \frac{\Gamma(c+k+N_{n-1}+n)\Gamma(c)}{\Gamma(c+k)\Gamma(c+N_{n-1}+n)},
}
which, using that~$x\Gamma(x)=\Gamma(x+1)$, easily simplifies to~\eq{17}.
\end{proof}

We use Lemma~\ref{lem2} to establish the almost sure behavior of~$\IE(D_{k,n}| T)$.

\begin{lemma}\label{lem3} Let~$\pi$ be a probability distribution on the non-negative integers with mean~$0<\mu\leq \infty$, and let~$\tau,\tau_1,\tau_2,\dots$ be a sequence of independent and identically distributed random variables with  distribution~$\pi$.
For~$i\geq 1$, let~$T_i=\sum_{j=1}^i \tau_i$, and for~$n\geq 0$, let~$N_n=\#\{i\geq 1: T_i \leq n\}$.
If either
\begin{enumerate}
\item[$(a)$] there is~$\eps>0$ such that~$\IE \tau^{1+\eps} <\infty$, or
\item[$(b)$] there is~$\eps>0$ such that~$\IP[\tau > n] \geq Cn^{-(1-\eps)}$ for~$n$ large enough,
\end{enumerate}
then, for any~$\alpha>0$ and~$\beta>0$,
there exists a (possibly random) positive number~$\chi(\alpha,\beta,\pi)$ such that, almost surely, 
\ben{\label{18}
n^{-\frac{\alpha\mu}{1+\mu}}\prod_{i=1}^{n}\bbklr{1+ \frac{ \alpha}{\beta+i+N_{i}}}\to \chi(\alpha,\beta,\pi),
\qquad\text{as~$n\toinf$.}
}
\end{lemma}

\begin{proof}
\textbf{Case~$\boldsymbol{(a)}$.}  Taking logarithm in \eq{18}, it is enough to show that
\be{
  \sum_{j=1}^n \log\bbklr{1+\frac{ \alpha}{\beta+j+N_{j}}} -	\frac{\alpha\mu}{1+\mu}\log n 
}
converges almost surely to a (possibly random) real number. Since both
\ben{\label{19}
	\log(n) - \sum_{j=1}^n \frac{1}{j}\qquad\text{and}\qquad \sum_{j=1}^n\log(1+x_j)-\sum_{j=1}^n x_j 
}
converge, provided 
\ben{\label{20}
\sum_{j=1}^\infty x_j^2<\infty,
}
it is enough to consider convergence of 
\besn{\label{21}
	& \sum_{j=1}^n\bbbklr{\frac{1}{\beta+j+N_j}-\frac{\mu}{(1+\mu)j}} \\
	& \qquad =   \sum_{j=1}^n\bbbklr{\frac{1}{\beta+j+N_j}-\frac{1}{j+N_j}}+ \sum_{j=1}^n\bbbklr{\frac{1}{j+N_j}-\frac{\mu}{(1+\mu)j}}.
}
Now, to prove \eq{20} for~$x_j = 1/(\beta+j+N_j)$, which justifies the second approximation in~\eq{19} and also convergence of the first sum on the right hand side of \eq{21}, we observe that, almost surely,
\ben{\label{22}
	\sum_{j=1}^\infty\bbklr{\frac{1}{\beta+j+N_j}}^2 \leq \sum_{j=1}^\infty{\frac{1}{j^2}} <\infty.
}
In order to prove convergence of the second sum on the right hand side of \eq{21}, we need a refined estimate for the renewal law of large numbers. 
Assume without loss of generality that~$\eps<1$. Let~$e^+_n = n\bklr{\mu^{-1}+n^{-\eps/2}}$ and~$E^+_{n} = \ceil{e^+_n}$, and observe that
\bes{
	\bbbklg{ \frac{N_n}{n}-\frac{1}{\mu} \geq n^{-\eps/2} } 
	&= \bklg{ N_n \geq e^+_n} \\
	& = \bklg{ T_{E^+_n} \leq n} \\
	& = \bbbklg{\frac{T_{E^+_n} - \mu E^+_n}{E^+_n} \leq \frac{n -  \mu E^+_n}{E^+_n} }.
}
Likewise, with~$e^-_n = n\bklr{\mu^{-1}-n^{-\eps/2}}$, $E^-_n = \floor{e^-_n}+1$, and~$n$ large
enough to ensure~$E^-_n>0$,
\bes{
	\bbbklg{ \frac{N_n}{n}-\frac{1}{\mu} \leq -n^{-\eps/2} } 
	&= \bklg{ N_n \leq e^-_n} \\
	& = \bklg{ T_{E^-_n} > n} \\
	& = \bbbklg{\frac{T_{E^-_n} - \mu E^-_n}{E^-_n} > \frac{n - \mu E^-_n}{E^-_n} }.
}
From this and the fact that~$\abs{n-\mu E^{\pm}_n}/E^{\pm}_n=\Theta( n^{-\eps/2})$, it is not difficult to see that there is a constant~$C>0$ such that
\besn{\label{23}
	\limsup_{n\geq 1} \bbbklg{ \bbabs{\frac{N_n}{n}-\frac{1}{\mu}} \geq n^{-\eps/2} }
	& \subset \limsup_{n\geq 1} \bbbklg{ \bbabs{\frac{T_n}{n}-\mu} \geq Cn^{-\eps/2} } \\
	& = \limsup_{n\geq 1} \bbbklg{ \frac{1}{n^{1-\eps/2}}\bbbabs{\sum_{i=1}^n (\tau_i-\mu)} \geq C }.
}
It follows from \cite[Theorem 17, p.~274, with~$a_n = n^{1-\eps/2}$]{Petrov1975}  and the fact that~$\IE \tau^{1+\eps}<\infty$ that the last event in \eq{23} has probability zero (alternatively use the Marcinkiewicz-Zygmund strong law of large numbers). Therefore, 
\ben{\label{24}
	\limsup_{n\toinf} n^{\eps/2}\bbabs{\frac{N_n}{n}-\frac{1}{\mu}} < \infty
}
 almost surely. Since
\bes{
	\sum_{j=1}^n\bbbabs{\frac{1}{j+N_j}-\frac{\mu}{(1+\mu)j}}
	& = \sum_{j=1}^n\bbbabs{\frac{1}{(1+N_j/j)j}-\frac{1}{(1+1/\mu)j}} \\
	& = \sum_{j=1}^n\frac{\abs{N_j/j-1/\mu}}{(1+N_j/j)(1+1/\mu)j} \\
	& \leq \sum_{j=1}^n\frac{j^{\eps/2}\abs{N_j/j-1/\mu}}{j^{1+\eps/2}}.
}
we conclude that, using~$\eq{24}$ and the fact that~$\sum_{j\geq 1}j^{-(1+\eps/2)}<\infty$, the last sum converges  almost surely as~$n\toinf$. 

\smallskip\noindent\textbf{Case~$\boldsymbol{(b)}$.} 
Following the proof of the~$\mu$ finite case (interpreting~$\infty/(1+\infty)$ as~$1$) up to and including~\eq{22},
it is sufficient to establish, as in~\eq{24}, that
$\limsup_{n\toinf}n^{\eps'}\frac{N_n}{n} < \infty$ almost surely for some~$\eps'>0$. Observe that 
\bes{
	\IP[T_n \leq n^{1+\eps}] & \leq \IP[\max\{\tau_1,\dots,\tau_n\} \leq n^{1+\eps}] 
	 = \bklr{\IP[\tau_1\leq n^{1+\eps}]}^n 
	 = \bklr{1-\IP[\tau_1> n^{1+\eps}]}^n \\
	& \leq \bbklr{1-\frac{C}{n^{(1-\eps)(1+\eps)}}}^n \\
	& = \bbklr{1-\frac{C}{n^{1-\eps^2}}}^n 
	 \leq \exp\bklr{-Cn^{\eps^2}}.
}
By Borel-Cantelli, 
\be{
	\IP\bbbkle{\limsup_{n\toinf}\bbklg{\frac{1}{n^{1+\eps/2}}T_n\leq n^{\eps/2}}}=0,
}
so that~$\frac{1}{n^{1+\eps/2}}T_n\toinf$ almost surely. By the usual relation between~$T_n$ and~$N_n$, this implies that, for~$\eps':=1-\frac{1}{1+\eps/2}>0$,
\be{	
	\limsup_{n\toinf}n^{\eps'}\frac{N_n}{n} < \infty
}
almost surely.
\end{proof}

The next result provides moment bounds for applying  dominated convergence to strengthen the convergence of Lemma~\ref{lem3}.

\begin{lemma}\label{lem4} Under the assumptions of Lemma~\ref{lem3}, if either
\begin{enumerate}
\item[$(a)$]~$\IE \tau^p <\infty$ for some~$p>1$, and~$1\leq k < (\frac{p}{2}-1)(\mu+1)$, or
\item[$(b)$]~$\IE \tau=\infty$ and~$k\geq 1$, 
\end{enumerate}
then, with~$D_{k,n}$ be as in Proposition~\ref{thm1},
\be{
	\limsup_{n\toinf}\frac{\IE D_{k,n}}{n^{k\mu/(1+\mu)}} < \infty.
}
\end{lemma}
\begin{proof} \textbf{Case~$\boldsymbol{(a)}$.} Using representation given by~\eq{17}, and the fact that~$\Gamma(x+k)\leq (x+k)^k\Gamma(x)$, we conclude that there is a constant~$C=C(b,w,k)$ such that
\ben{\label{25}
	\IE\klr{D_{k,n}|T} \leq C(c+k+N_{n-1}+n)^k\prod_{j=1}^{N_{n-1}}\bbklr{1-\frac{k}{c+j-1+T_j+k}},
}
where we write~$c=b+w$ to shorten formulas. Let 
~$A_n$ be the event that~$\abs{N_{n-1}-(n-1)/\mu}\leq  u_n :=(n-1)/(\mu+1)-1$. We have
\ben{
\IE D_{k,n} = \IE \klr{D_{k,n} \I[ A_n]}
+\IE \klr{D_{k,n} \I[A_n^c]}, \label{26}
}
and we show that both terms on the right hand side of~\eq{26} are~$\bigo(n^{k\mu/(\mu+1)})$.
Below it is important to notice that~$A_n$ is in the sigma-algebra generated by~$T$.
Now, for the first term of~\eq{26}, use the expression~\eq{25} to find that, under the event~$A_n$,
\bes{
\IE (D_{k,n}|T) \I[A_n]& \leq C (c+k+n/\mu+u_n+n)^k 
 \prod_{j=1}^{\phi_n}\left(1-\frac{k}{c+j-1+T_j+k}\right),
}
where~$\phi_n=\lfloor \frac{n-1}{\mu}-u_n\rfloor$,
and note that~$1\leq\phi_n =\Theta(n)$ by our definition of~$u_n$. 
Since
\be{
(c+k+n/\mu+u_n+n)^k=\bigo(n^k),
}
it is sufficient to show that 
\be{
\IE \prod_{j=1}^{\phi_n}\left(1-\frac{k}{c+j-1+T_j+k}\right)=\bigo\bklr{n^{-k/(1+\mu)}}.
}
Let~$1/2< \alpha< 1$ and set~$U_\alpha:=\sup\{j\geq1: T_j> j\mu+j^{\alpha}\}$ to be 
the last time that the centered random walk~$(T_j-j \mu)_{j\geq0}$ is larger than~$j^\alpha$; note that~$U_\alpha$ is almost surely finite by the law of the iterated logarithm.
Defining the empty product to be one, we have
\besn{
	&\IE \prod_{j=1}^{\phi_n}\left(1-\frac{k}{c+j-1+T_j+k}\right) \\
	&\qquad\leq \IE \prod_{j=U_\alpha}^{\phi_n}\left(1-\frac{k}{c+j-1+T_j+k}\right) 
	\leq  \IE \prod_{j=U_\alpha}^{\phi_n}\left(1-\frac{k}{c+j-1+j \mu + j^{\alpha}+k}\right) \\
	&\qquad\leq \prod_{j=1}^{\phi_n}\left(1-\frac{k}{c+j-1+j \mu + j^{\alpha}+k}\right) 
	\IE \prod_{j=1}^{U_\alpha}\left(1+\frac{k}{j+j \mu + j^{\alpha}}\right).
\label{27}
}
We show the first product of~\eq{27} is~$\bigo(n^{-k/(1+\mu)})$, and the second is bounded.
Taking logarithm in the first product, we claim
\be{
\sum_{j=1}^{\phi_n} \log\left(1-\frac{k}{c+j-1+j\mu+j^\alpha+k}\right) + \frac{k}{1+\mu} \log(n)
}
converges as~$n\to\infty$. 
Since~$\log(\phi_n)-\log(n)$ converges, we can replace~$\log(n)$ with~$\log(\phi_n)$.
Now, similar to the proof of Lemma~\ref{lem3},
since 
\be{
 \sum_{j=1}^n\log(1+x_j)-\sum_{j=1}^n x_j 
}
converges provided~$\sum_{j=1}^\infty x_j^2<\infty$, it is enough to consider the convergence of 
\bes{
 & -\sum_{j=1}^{\phi_n} \frac{1}{c+j-1+j\mu+j^\alpha+k} +\frac{1}{1+\mu}\log(\phi_n)\\
 &\qquad= \sum_{j=1}^{\phi_n} \left(\frac{1}{j(1+\mu)}- \frac{1}{c+j(1+\mu)+j^\alpha+k-1}\right) +\bigo(1) \\
 &\qquad= \sum_{j=1}^{\phi_n} \frac{c+j^\alpha+k-1}{(1+\mu) j(c+j(1+\mu)+j^\alpha+k-1)}+\bigo(1) \\
 &\qquad \leq C \sum_{j=1}^{\phi_n} j^{-2+\alpha} +\bigo(1);
  }
here~$C$ is some constant and the sum is convergent since~$\alpha<1$. For the second product of~\eq{27},
easy variations of the arguments above (or in the proof of Lemma~\ref{lem3}) show that there is a constant~$C$ (depending on~$k$ and~$\mu$) 
such that for any~$t\geq 1$,
\be{
 \prod_{j=1}^{t}\left(1+\frac{k}{j+j \mu + j^{\alpha}}\right)\leq
 \prod_{j=1}^{t}\left(1+\frac{k}{j+j \mu }\right) \leq C t^{k/(\mu+1)}.
}
Substituting~$t=U_\alpha$, it is enough to show that~$\IE U_\alpha^{k/(\mu+1)} <\infty$.
We choose~$\alpha<1$ close enough to one to ensure~$p>\bklr{\frac{k}{\mu+1}+1}/\bklr{\alpha-1/2}$
and find
\besn{ \label{28}
\IP\kle{U_{\alpha}\geq x}&=\IP\bkle{\cup_{j\geq x}\{T_j> j\mu+j^{\alpha}\}} \\
	&\leq \sum_{j\geq x} \IP\bkle{\abs{T_j- j\mu}>j^{\alpha}} \\
	&\leq \sum_{j\geq x}\frac{ \IE \abs{T_j- j\mu}^p}{j^{p \alpha}} \\
	&\leq  C_p \IE \abs{\tau-\mu}^p\sum_{j\geq x} j^{-p(\alpha-1/2)}, 
}
where the last inequality follows from Lemma~\ref{lem6} below. Using \eq{28}, we obtain
\bes{
\IE U_\alpha^{k/(\mu+1)}& =\frac{k}{\mu+1}\int_{0}^\infty  x^{k/(\mu+1)-1}\IP\kle{U_\alpha> x} dx \\	
		&=\frac{k}{\mu+1}\int_{0}^\infty  x^{k/(\mu+1)-1}\IP\kle{U_\alpha\geq\floor{x}+1} dx \\	
		&\leq C_p \frac{ k\IE \abs{\tau-\mu}^p}{\mu+1} \int_{0}^\infty  x^{k/(\mu+1)-1} \sum_{k \geq \floor{x}+1} k^{-p(\alpha-1/2)} dx \\
		&\leq   \frac{ k C_p \IE \abs{\tau-\mu}^p}{(\mu+1)(p(\alpha-1/2)-1)} \int_{0}^\infty  x^{k/(\mu+1)-1} \floor{x}^{-p(\alpha-1/2)+1}dx \\
		&\leq \frac{k C_p \IE \abs{\tau-\mu}^p}{(\mu+1)(p(\alpha-1/2)-1)} \int_{0}^\infty  x^{k/(\mu+1)-p(\alpha-1/2)}dx<\infty;
	}
the finiteness is by the assumption that~$p>\bklr{\frac{k}{\mu+1}+1}/\bklr{\alpha-1/2}$.

For the second term of~\eq{26}, first note that
the term~\eq{16} is decreasing in the~$N_i$, so the conditional expectation (without the indicator) is almost surely bounded
\ben{\label{29}
\IE(D_{k,n}|T)\leq \frac{\Gamma(w+k)}{\Gamma(w)} \prod_{j=0}^{n-1}\frac{c+k+j}{c+j} 
	\leq \frac{\Gamma(w+k)}{\Gamma(w)}\cdot\frac{\Gamma(c)}{\Gamma(c+k)}\cdot\frac{\Gamma(c+k+n)}{\Gamma(c+n)} =\bigo(n^k).
}
Now noting that~$A_n^c$ is in the sigma-algebra generated by~$T$, it is enough to show that~$\IP[A_n^c]=\bigo\bklr{n^{-k/(\mu+1)}}$.
Denoting~$\omega_n:=\lceil(n-1)/\mu+u_n \rceil$ and using the moment bound
of Lemma~\ref{lem6} below, we have 
\bes{
\IP[A_n^c] &= \IP\bkle{\abs{N_{n-1}-(n-1)/\mu} >u_n}\\
	&\leq \IP[T_{\omega_n}< n-1]+ \IP[T_{\phi_n}\geq n-1] \\
	&\leq \IP[T_{\omega_n}-\mu \omega_n< n-1 -\mu\omega_n]+ \IP[T_{\phi_n}-\mu \phi_n \geq n-1-\mu\phi_n] \\[0.3\baselineskip]
	&\leq C_{2k/(\mu+1)} \IE \abs{\tau-\mu}^{2k/(\mu+1)} 
	\left(\frac{\omega_n^{k/(\mu+1)}}{(\mu\omega_n-n-1)^{2k/(\mu+1)}}+\frac{\phi_n^{k/(\mu+1)}}{(n-1-\mu\phi_n)^{2k/(\mu+1)}}\right). 
}
But since~$\phi_n=(n-1)/\mu-u_n-\eps_1$ and~$\omega_n=(n-1)/\mu+u_n+\eps_2$ for some~$\eps_1,\eps_2\in[0,1)$, the last expression
is~$\bigo(n^{-k/(\mu+1)})$, as desired.

\medskip
\noindent\textbf{Case~$\boldsymbol{(b)}$.} Since~$X_n\leq n+w$, then~$D_{k,n} \leq (n+w+k)^k$, and so~$\IE D_{k,n}=\bigo(n^k)$,
as required.
\end{proof}

We now give the proof of Proposition~\ref{thm1} and then note that Theorem~\ref{THM1}
easily follows from that result.

\begin{proof}[Proof of Proposition~\ref{thm1}]
\textbf{Case~$\boldsymbol{(a)}$.} 
Lemma~\ref{lem3}(a) applied to \eq{16} of Lemma~\ref{lem2} implies that 
$n^{-k\mu/(\mu+1)} \IE [D_{k,n}| T]$ converges almost surely to a positive random variable~$\chi(k,b+w,\pi)$. Using Lemma~\ref{lem4} and dominated convergence (see, for example, \cite[Exercise~3.2.5]{Durrett2010}), it follows that 
$n^{-k\mu/(\mu+1)} \IE D_{k,n} \to \IE\chi(k,b+w,\pi)$.

\smallskip\noindent\textbf{Case~$\boldsymbol{(b)}$.} Analogous to Case~$(a)$.
\end{proof}

\begin{proof}[Proof of Theorem~\ref{THM1}] In both cases~$(a)$ and~$(b)$, Lemma~\ref{lem3} implies that~$X_n$ converges to infinity almost surely. Hence, regular and factorial moments are asymptotically equivalent, and Theorem~\ref{THM1} follows directly from Proposition~\ref{thm1} with the same constants~$m_k(b,w,\pi)$.
\end{proof}

The following lemma is given in \cite[16, Page 60]{Petrov1975}
where it is attributed to \cite{Dharmadhikari1969}.

\begin{lemma}\label{lem6}
Let~$Y_1,\ldots, Y_n$ be independent random variables such that
for~$i=1,\ldots,n$,
$\IE Y_i=0$  and~$\IE \abs{Y_1}^p<\infty$, and let~$ S_n=\sum_{i=1}^n Y_i$. Then
\be{
\IE \abs{ S_n}^p \leq C_p n^{p/2-1} \sum_{i=1}^n \IE \abs{Y_i}^p,
}
where
\be{
C_p=\frac{1}{2} p(p-1) \max (1, 2^{p-3}) \left( 1+ \frac{2}{p} K_{2m}^{(p-2)/2m}\right),
}
and the integer~$m$ satisfies~$2m \leq p < 2m +2$, and 
\be{
K_{2m}=\sum_{r=1}^m \frac{r^{2m-1}}{(r-1)!}.
}
\end{lemma}

\section{Proof of Theorem~\ref{THM2}}\label{sec4}
 
We prove the convergence first for~$b=1$,
and then the general case follows easily from an auxiliary 
P\'olya urn argument. 

Recall that~$W_n=n^{-\mu/(\mu+1)}X_n$ and we want to
derive the distributional limit of the sequence~$W_n$.
The method of proof is to show tightness of the sequence~$\law(W_n)$,
and then use the characterizing properties of the limit given by Proposition~\ref{prop1}
to prove convergence. To simplify notation, for a probability distribution~$\pi=(\pi_k)_{k\geq 0}$, let 
~$\pi^* = (\pi^*_k)_{k\geq 1}$ be the distribution defined as~$\pi^*_k = \pi_{k-1}$ for~$k\geq 1$. Moreover, let~$\supp(\pi^*) = \{k\geq 1:\pi^*_k>0\}$ be the support of~$\pi^*$.
 
\begin{lemma}\label{lem7}
Let~$(W_n)_{n\geq0}$ be a sequence of non-negative random variables. Let~$w>0$, let~$\pi =(\pi_k)_{k\geq 0}$ be a probability distribution, and let~$(m_k)_{k\in\supp(\pi^*)}$ be positive numbers. If
\begin{itemize}
\item[(i)] for each~$k\in\supp(\pi^*)$ there is~$\eps>0$ such that\/~$\limsup_{n\to\infty} \IE W_n^{k+\eps}<\infty$,
\item[(ii)]~$\lim_{n\toinf}\IE W_n^k = m_k$ for all~$k\in\supp(\pi^*)$, and
\item[(iii)] for each~$n$ there is a coupling~$\bklr{W_n, B_n W_n\s{\pi^*}}$, where~$W_n\s{\pi^*}$
has the~$\pi^*$-power-bias distribution of~$W_n$ defined
through~\eq{5}, where~$B_n\sim\Beta(w,1)$ is
independent of~$W_n\s{\pi^*}$, and such that as~$n\to\infty$,
\be{
\law\bklr{W_n - B_{n} W_n\s{\pi^*}}\to0,
}
\end{itemize}
then~$\law(W_n)\to \ul\bklr{w;(a_k)_{k\geq 1}}$ as~$n\to\infty$, where~$a_k = \pi_k^*/m_k$ for~$k\in\supp(\pi^*)$ and~$a_k=0$ for~$k\not\in\supp(\pi^*)$.
\end{lemma}
\begin{proof}
From~$(i)$ we conclude that~$\limsup_{n\to\infty}\IE W_n <\infty$, so that the sequence~$(\law(W_n))_{n\geq1}$ is
tight. Thus, we assume that~$\law(W_n)\to \law(W)$ and show that this implies~$W\sim\ul\bklr{w;(a_k)_{k\geq1}}$. 
As per Proposition~\ref{prop1}, it is enough to show that
\ben{\label{30}
(a)\enskip \IE W^k = m_k, \quad k\in \supp(\pi^*),
\qquad\text{and}\qquad
(b)\enskip\law(W) = \law\bklr{V_w W\s{\pi^*}},
}
where~$V_w\sim\Beta(w,1)$ is independent of~$W\s{\pi^*}$.
Now, $(i)$, $(ii)$ and dominated convergence (see, for example, \cite[Exercise~3.2.5]{Durrett2010}) imply
that~$\IE W^k=m_k$ for~$k\in\supp(\pi^*)$, which is~$(a)$. Using~$(iii)$ and Slutsky's theorem, we conclude that~$\law\bklr{B_n W_n\s{\pi^*}}\to \law(W)$. 
 But we also have
that~$\law\bklr{B_n W_n\s{\pi^*}}\to \law\bklr{V_w W\s{\pi^*}}$. Indeed, first show~$\law(W_n\s{\pi^*})\to \law(W\s {\pi^*})$:
for bounded and continuous~$f$,
\ben{\label{31}
\IE f\bklr{W_n\s{\pi^*}}=\sum_{k\geq 1} \pi^*_{k} \frac{\IE \bklr{W_n^k f(W_n)}}{\IE W_n^k} \leq \norm{f}_{\infty},
}
and by~$(i)$ and dominated convergence, $\IE \bklr{W_n^k f(W_n)}\to \IE\bklr{ W^k f(W)}$. So by bounded convergence
applied to the sum in~\eq{31}, as~$n\to\infty$,
\be{
\IE f\bklr{W_n\s{\pi^*}}=\sum_{k\geq 1} \pi^*_k \frac{\IE \bklr{W_n^k f(W_n)}}{\IE W_n^k}\longrightarrow\sum_{k\geq 1} \pi^*_k \frac{\IE\bklr{ W^k f(W)}}{\IE W^k} = \IE f\bklr{W\s{\pi^*}}.
}
Moreover, it's obvious that~$\law(B_n)\to \law(V_w)$, and, using independence of the relevant pairs of variables,
$\law\bklr{(B_n, W_n\s{\pi^*})}\to \law\bklr{(V_w, W\s{\pi^*})}$.
Now  the continuous mapping theorem implies~$\law\bklr{B_n W_n\s{\pi^*}}\to \law\bklr{V_w W\s{\pi^*}}$,
as desired. Combining these facts, we find
that~$(b)$ also holds, and it follows from Proposition~\ref{prop1} that~$W\sim\ul\bklr{w;(a_k)_{k\geq 1}}$.
\end{proof}

Our strategy to proof Theorem~\ref{THM2} is to apply Lemma~\ref{lem7} to 
\be{
	W_n = \frac{X_n}{n^{\mu /(1+\mu)}}.
}
Assuming that~$\pi$ has all positive moments finite or infinite mean, and then choosing~$m_k$ as in \eq{7}, 
we conclude that~$(i)$ and~$(ii)$ of Lemma~\ref{lem7} are satisfied. Thus, it is sufficient to show~$(iii)$ of Lemma~\ref{lem7}. We develop the coupling 
of~$W_n$ to a variable distributed
as~$V_{w,n} W_n\s{\pi^*}$ over a series of lemmas, working first on~$W_n\s{\pi^*}$.
Denote the rising factorial 
\be{
	x\rf k:= x(x+1)\dots(x+k-1).
}

\begin{definition}
Let~$\psi$   be a probability distribution concentrated on the positive integers, and let~$W$ be a positive random variable such that~$\IE W^{k}<\infty$,
for all~$k$ in~$\supp(\psi)$. 
A random variable~$W\rfb \psi$ is said to have the \emph{$\psi$-rising-factorial-bias
distribution of~$W$} if
\ben{\label{32}
	\IE f\bklr{W\rfb{\psi}} = \sum_{k\in\supp(\psi)} \psi_k  \frac{\IE\bklr{ W\rf k f(W)}}{\IE W\rf {k}}
}
for all~$f$ for which the expectation on the right hand side exists.
If~$\psi_k=1$ for some~$k\geq1$, then we simply write
$W\rfb{k}$ to denote~$W\rfb{\psi}$.
 \end{definition}

The next lemma relates the~$\pi^*$-rising-factorial-bias distribution of~$W_n$ to its~$\pi^*$-power-bias distribution.
\begin{lemma}\label{lem8}
Let~$b$ and~$w$ be positive integers, let~$\pi$ be a distribution on the non-negative integers, let~$\tau\sim\pi$, and assume that either
\begin{itemize}
\item[$(a)$]~$\IE \tau^p< \infty$ for all~$p\geq 1$,
or
\item[$(b)$] there is~$\eps>0$ such that~$\IP[\tau > n]\geq C n^{-(1-\eps)}$ for~$n$ large enough.
\end{itemize}
Let~$X_n\sim\urnlaw(n,\pi,b,w)$, and let~$X_n\s{\pi^*}$, respectively~$X_n\rfb{\pi^*}$, 
have the~$\pi^*$-power-bias, respectively the~$\pi^*$-rising-factorial-bias distribution of~$X_n$. 
Then
\be{
\dtv\bbklr{\law\bklr{X_n\s{\pi^*}}, \law\bklr{X_n\rfb{\pi^*}}}\to 0\qquad\text{as~$n\toinf$.}
}  
\end{lemma}
\begin{proof}
We show that for each fixed~$k\geq 1$,
\ben{\label{33}
\dtv\bbklr{\law\bklr{X_n\s{k}}, \law\bklr{X_n\rfb{k}}}\to 0,
}
from which the lemma follows by bounded convergence and the fact that in general,
for random variables~$(X,Y,U)$ defined on the same probability space,
\be{
\dtv(\law(X), \law(Y))\leq \IE \dtv\bklr{\law( X|U), \law(Y|U)}.
}
Both~$X_n\s{k}$ and~$X_n\rfb{k}$ have densities with respect to~$X_n$, and so 
\ban{
&2\dtv\bbklr{\law\bklr{X_n\s{k}}, \law\bklr{X_n\rfb{k}}}\notag \\
&\qquad\quad=\sum_{j\geq0} \IP(X_n=j)\bbbabs{\frac{j^{\overline k} }{\IE D_{k,n}}-\frac{j^k}{\IE X_n^k}}\notag \\
&\qquad\quad=\sum_{j\geq0}\frac{\IP(X_n=j)}{\IE D_{k,n}}
	\bbbbabs{j^k\bbbklr{1-\frac{\IE D_{k,n}}{\IE X_n^k}} + \sum_{i=0}^{k-1} { k \brack i} j^i}\notag \\
&\qquad\quad\leq\frac{\IE X_n^k}{\IE D_{k,n}}\bbbabs{1-\frac{\IE D_{k,n}}{\IE X_n^k}}
+ \frac{1}{\IE D_{k,n}}\sum_{i=0}^{k-1} { k \brack i} \IE X_n^i,\label{34}
}
where the~${k \brack i}$ are unsigned Stirling numbers of the first kind. But due
to the moment or tail assumptions on~$\pi$, Proposition~\ref{thm1} implies that 
$\IE D_{i,n} =\Theta(n^{i\mu/(\mu+1)})$ for all~$i=1,\ldots$. Therefore, $\IE X_n^i$ must be of the same order, and moreover,
$\IE D_{k,n}/\IE X_n^k\to1$ as~$n\to\infty$. Applying these facts with~\eq{34} implies the lemma.
\end{proof}

We use the rising factorial bias distribution because it can be connected back to our (unbiased) 
urn models.
\begin{lemma}\label{lem9}
Let~$b$ and~$w$ be positive integers, let~$\pi$ be a distribution on the non-negative integers, let~$\tau\sim\pi$, and assume that either~$(a)$ or~$(b)$ from Proposition~\ref{thm1} holds. Let~$X_n\sim\urnlaw(n,\pi,b,w)$, and let~$\bklr{Y_n(k)}_{k\in\supp(\pi^*),n\geq 0}$ be a family of random variable such that~$Y_n(k) + k\sim\urnlaw(n,\pi,b,w+k)$.
If~$X_n\rfb{\pi^*}$ has the~$\pi^*$-rising-factorial-bias distribution of~$X_n$, 
then
\be{
\dtv\bbklr{\law\bklr{X_n\rfb{\pi^*}},\law\bklr{Y_n(\tau^*)}}\to 0 \qquad \text{as~$n\to\infty$},
}
where~$\tau^*\sim\pi^*$ is independent of~$\bklr{Y_n(k)}_{k\in\supp(\pi^*),n\geq 0 }$.
\end{lemma}
\begin{proof} 
As in the start of the proof of Lemma~\ref{lem8},
it is sufficient to show that for each~$k\geq 1$,
\ben{\label{35}
\dtv\bbklr{\law\bklr{X_n\rfb{k}},\law\bklr{Y_n(k)}}\to 0,
}
For the remainder of the proof, we keep~$k\geq 1$ fixed and, thus, drop it from out notation. We define three urn process, coupled together through the immigration times in the following way. 

First, let~$X=(X_0,X_1,X_2,\dots)$ be a realisation of the immigration urn model starting with~$b$~black and~$w$~white balls, and with immigration distribution~$\pi$; let~$T=(T_1, T_2,\ldots)$ be the corresponding sequence of arrival times of immigrating black balls. 
Second, let~$\tilde X = (\tilde X_1,\tilde X_2,\dots)$ be a sequence of random variables such that, given~$T$,
\ben{\label{36}
\IP[\tilde X_n=j| T]= \frac{j^{\overline{k}} \IP[X_n=j| T]}{\IE (D_{k,n}|T) };
}
that is, $\tilde X_n$ has the k-rising-factorial-bias distribution of~$X_n$ conditional on~$T$.
Third, let~$Y_1,Y_2,\dots$ be a realisation of the urn model starting with~$b$ black and~$w+k$ white balls, where the immigration times are also~$T$. We note that, given~$T$, the joint distribution of the three processes is not going to be relevant.

Applying representation \eq{16} from Lemma~\ref{lem2} to~$Y_n+k\sim\urnlaw(n,\pi,b,w+k)$, we obtain that, for any~$l\geq 1$, 
\ba{
&\IE\bbbklg{\,\prod_{j=0}^{l-1} (Y_n+k+j)\,\bbbmid\, T} \\
&\quad = \frac{\Gamma(w+k+l)}{\Gamma(w+k)}
			\prod_{j=0}^{n-1}\frac{b+w+k+l+j+N_j}{b+w+k+j+N_j}\\
&\quad = \frac{\Gamma(w)}{\Gamma(w+k)}
			\prod_{j=0}^{n-1}\frac{b+w+j+N_j}{b+w+k+j+N_j}\times \frac{\Gamma(w+k+l)}{\Gamma(w)}
			\prod_{j=0}^{n-1}\frac{b+w+k+l+j+N_j}{b+w+j+N_j} 
			\\
&\quad = \frac{1}{\IE\bklr{D_{k,n}\bmid T}}
		\times \IE\bbbklg{\,\prod_{j=0}^{k+l-1}(X_n+j)\,\bbbmid\,T}
			\\
&\quad = \frac{1}{\IE\bklr{D_{k,n}\bmid T}}
		\times \IE\bbbklg{\,\prod_{j=0}^{k-1}(X_n+j)\times\,\prod_{j=0}^{l-1}(X_n+k+j)\,\bbbmid\,T}
			\\
&\quad = \IE\bbbklg{\, \prod_{j=0}^{l-1} (\tilde X_n+k+j)\,\bbbmid\, T}.
}
Taking expectations on both sides of the previous display and using the method of moments, we deduce that, in fact, $\law(Y_n)=\law(\tilde X_n)$ for all~$n\geq 0$. Thus, we have reduced the problem to showing that, as~$n\to\infty$,
\ben{\label{38}
\dtv\bklr{\law\bklr{X_n\rfb{k}},\law\bklr{\tilde X_n}}\to 0.
}
Using \eq{32} and \eq{36}, we find
\besn{\label{39}
2\dtv\bklr{\law\bklr{X_n\rfb{k}},\law\bklr{\tilde X_n}}
&=\sum_{j\geq0}\bbbbabs{ \IE  \bbbbklg{\frac{j^{\overline k}\IP[X_n=j| T]}{\IE\klr{D_{k,n}|T}}-\frac{j^{\overline k}\IP[X_n=j| T]}{\IE D_{k,n}}}}  \\
&\leq \IE\bbbbklg{\bbbabs{1-\frac{\IE\klr{D_{k,n}|T}}{\IE D_{k,n}} }\sum_{j\geq0}\frac{j^{\overline k}\IP[X_n=j| T]}{\IE\klr{D_{k,n}|T}}}\\
&=\IE\bbbabs{1-\frac{\IE\klr{D_{k,n}|T}}{\IE D_{k,n}}}.
}
Now, by Lemma~\ref{lem3} and Proposition~\ref{thm1}, we have that, almost surely,
\be{
\frac{\IE\klr{D_{k,n}|T}}{\IE D_{k,n}}\to 1\qquad\text{as~$n\toinf$.}
}
Moreover, Jensen's inequality implies~$\IE(\IE\klr{D_{k,n}|T}^2)\leq \IE D_{k,n}^2\leq\IE D_{2k,n}$, and thus, again by Proposition~\ref{thm1},
\be{
\sup_{n\geq 1}\frac{\IE D_{2k,n}}{\left(\IE D_{k,n}\right)^2}<\infty.
}
Hence, by dominated convergence, the right hand side of \eq{39} tends to zero, which concludes the proof.
\end{proof}

The next two lemmas move us from~$Y_n(\tau^*)$ defined in Lemma~\ref{lem9}
to a variable that will be used as a surrogate for~$W_n\s{\pi^*}$.

\begin{lemma}\label{lem10}
Let~$w$ be a positive integer, let~$\pi$ be a probability distribution on the non-negative integers. Let~$(Y_n(k))_{k\in\supp(\pi^*)}$ be a family of random variables such that~$Y_n(k)+k\sim\urnlaw(n,\pi,1,w+k)$ for~$k\in \supp(\pi^*)$. Moreover, let~$ \tilde X=( \tilde X_0, \tilde X_1,\dots)$ be a realisation of an immigration urn process with immigration distribution~$\pi$, starting with zero black balls and~$w+1$ white balls, so that~$\tilde X_n \sim \urnlaw(n,\pi,0,w+1)$. Let~$\tilde \tau$ be time of the first arrival in the urn process~$\tilde X$, and let~$\tilde Y_n = \tilde X_{n+\tilde \tau}-\tilde \tau-1$.  
Then
\be{
\law\bklr{Y_n(\tau^*)}=\law\bklr{\tilde Y_n},
}
where~$\tau^*\sim\pi^*$ is independent of~$(Y_n(k))_{k\geq 1}$.
\end{lemma}
\begin{proof}
Consider the urn process~$\tilde X$. Since there are no black balls in the urn initially,
the first~$\tilde \tau$ draws all come up white and so~$\tilde \tau$ white balls are added to the urn,
$\tilde \tau$ steps elapse, and one black ball is added. At this point there are~$w+\tilde \tau+1$
white balls in the urn and~$1$ black ball. Thus we find that 
$\urnlaw(n+\tilde \tau,\pi,0,w+1)=\urnlaw(n,\pi,1,w+\tilde \tau+1)$, which is exactly
the statement of the lemma.
\end{proof}

\begin{lemma}\label{lem11}
Let~$w$ be a positive integer, let~$\pi$ be a probability distribution on the non-negative integers, let~$\tilde X=(\tilde X_0,\tilde X_1,\dots)$, $\tilde\tau$ and~$(\tilde Y_n)_{n\geq 1}$ be defined as in Lemma~\ref{lem10}. Then,
\be{
\frac{1}{n^{\mu/(\mu+1)}}\bklr{\tilde Y_n - (\tilde X_n-w-1)} 
\convprob 0.
}
\end{lemma}
\begin{proof}
The only difference between the two variables is the number of steps the process is run, and the shifts~$\tilde \tau$ and~$w-1$. Since, at each time step, the number of white balls in the urn increase by at most one, we have
\be{
\babs{\tilde Y_n-(\tilde X_n-w-1)}=\babs{\tilde X_{n+\tilde \tau}-\tilde\tau - 1-(\tilde X_n-w-1)}\leq 2\tilde\tau +w.
}
Divided by the scaling~$n^{\mu/(\mu+1)}$, the right hand side tends to zero in probability.
\end{proof}

The previous lemmas imply we can use~$n^{-\mu/(\mu+1)}(\tilde X_n-w-1)$ as a surrogate
for~$W_n\s{\pi^*}$, and the next result shows how to relate this variable back to the original
$W_n$ using a classical P\'olya urn.
\begin{lemma}\label{lem12} Let~$w$ be a positive integer and let~$\pi$ be a probability distribution on the non-negative integers. Let~$\tilde X=(\tilde X_1,\tilde X_2,\dots)$ be as in Lemma~\ref{lem10}, and let~$\bklr{Q_w(n)}_{n\geq0}$ be the number of white balls in a classical P\'olya urn sequence
started with~$1$ black ball and~$w$~white balls. Then
\be{
	Q_w\bklr{\tilde X_n- w-1}\sim\urnlaw(n,\pi,1,w).
}
\end{lemma}
\begin{proof}
Start with an urn having~$w$ white balls, $1$ gray ball, and~$0$ black balls.
The urn follows the rules of a classical P\'olya urn with three colors, but at
the arrival times~$T_1, T_2,\ldots,$ driven by~$\pi$, a black ball is added to the urn.
It is clear that~$\tilde X_n-w-1$ equals the number of times a gray or white 
ball is drawn after~$n$ steps in this urn process, and each time a gray or white ball is drawn,
the chance of it being white is proportional to the number of white balls in the urn at that moment, just
as in a classical P\'olya urn. So~$Q_w\bklr{\tilde X_n-w-1}$ is distributed as the number of white balls in the described urn after~$n$ steps, but this distribution is exactly 
$\urnlaw(n,\pi,1,w)$ since the~$1$ gray ball can now be viewed as a ``black" ball. 
\end{proof}

To get to the beta variable~$V_w$ in the coupling (and to transfer to the general~$b>1$ case),
we need the result of \cite[Lemma~2.3]{Pekoz2014a}, which provides a close coupling of a classical 
P\'olya urn to its beta limit. Denote by~$\P(n,b,w)$ the law of the number
of white balls in a classical P\'olya urn started with~$b$ black and~$w$ white balls after~$n$
 draws and replacements.
\begin{lemma}\label{lem13}
Let~$\beta$, $\omega$ and~$k$ be positive integers. There is a coupling~$\bklr{Q_{\beta,\omega}(n),V_{\beta,\omega}}$
with~$Q_{\beta,\omega}(n)\sim \P(n,\beta,\omega)$ 
and~$V_{\beta,\omega}\sim \Beta(\omega,\beta)$, 
such that, almost surely,
\be{
\abs{Q_{\beta,\omega}(n)-n V_{\beta,\omega}} <\beta(4\omega+\beta+1).
}
\end{lemma}

We are now in position to complete the proof of Theorem~\ref{THM2}.

\begin{proof}[Proof of Theorem~\ref{THM2}]
We first show the result for~$b=1$.
Let~$W_n=n^{-\mu/(\mu+1)} X_n$.
By~\eq{7} and because either~$(a)$ or~$(b)$ is satisfied, there is a sequence~$m=(m_1,m_2,\ldots)$ such that
$\IE W_n^k\to m_k$.
We want to show that~$\law(W_n)\to\ul\bklr{w;(a_k)_{k\geq 1}}$, and we do so by showing~$(i)$, $(ii)$ and~$(iii)$ of Lemma~\ref{lem7}.
By~\eq{7}, $(i),(ii)$ easily follow. To show~$(iii)$, 
Lemmas~\ref{lem8}--\ref{lem11}
imply that we can couple variables~$\bklr{W_n\s{\pi^*}, n^{-\mu/(\mu+1)}\hat X_n}$,
where~$W_n\s{\pi^*}$ has the~$\pi^*$-power-bias distribution of~$W_n$ 
and~$\hat X_n=\tilde X_n-w-1$, such that
\ben{\label{40}
\bklr{W_n\s{\pi^*}-n^{-\mu/(\mu+1)} \hat X_n} \convprob 0\qquad\text{as~$n\to\infty$.}
}
Moreover, Lemma~\ref{lem12} implies
\be{
\hat W_n:=\frac{Q_w\bklr{\hat X_n}}{n^{\mu/(\mu+1)}}\eqlaw W_n,
}
where~$Q_w(n)$ is defined as in Lemma~\ref{lem12}.
By Lemma~\ref{lem13}, there is a coupling 
$(Q_w\bklr{\hat X_n}, V_w \hat X_n)$ with~$V_w\sim\Beta(w,1)$ independent of~$\hat X_n$ and
such that
\be{
\babs{Q_w\bklr{\hat X_n}- V_w \hat X_n}< w+1
} 
almost surely. From these last two displays, we have 
a coupling~$\bklr{V_w W_n\s{\pi^*}, \hat W_n}$ with the appropriate marginals satisfying
\bes{
\babs{V_wW_n\s{\pi^*}-\hat W_n}&\leq \babs{V_wW_n\s{\pi^*}-n^{-\mu/(\mu+1)} V_w\hat X_n}
						+\babs{n^{-\mu/(\mu+1)} V_w\hat X_n-\hat W_n} \\
							&< \babs{W_n\s{\pi^*}-n^{-\mu/(\mu+1)}  \hat X_n}
						+n^{-\mu/(\mu+1)} (w+1),
}
which according to~\eq{40} tends to zero in probability, as desired.
Finally, the convergence of the moments of~$W_n$ to those of its limit follows since~$\IE W_n^k \to m_k<\infty$ for all~$k\geq 1$; this implies \eq{9}.

For the general case~$b>1$, let~$X_n\sim \urnlaw(n,\pi,b,w)$ and~$X'_n \sim \urnlaw(n,\pi,1,w+b-1)$. We show that
\ben{\label{41}
\law\klr{X_n}=\mathcal{P}\bklr{{\textstyle{b-1\atop w}};X_n'-(b+w-1)},
}
and then the result follows easily from Lemma~\ref{lem13}.
To establish~\eq{41}, consider an urn that at step zero
has~$w$ white balls, $b-1$ gray balls, and~$1$~black ball. 
The urn follows the rules of a classical P\'olya urn but at
the arrival times~$T_1, T_2,\ldots,$ driven by~$\pi$, a black ball is added to the urn.
It is clear that~$X_n'-(b+w-1)$ is distributed as the number
of times a gray or white ball is drawn
after~$n$ steps in this urn process, and each time a gray or white ball is drawn,
the chance it is white is proportional to the number of white balls in the urn at that moment, just
as in a classical P\'olya urn. So
$\mathcal{P}\bklr{{\textstyle{b-1\atop w}};X_n'-(b+w-1)}$
is the distribution of the number of white balls in the urn process after~$n$ steps,
and this is exactly~$\urnlaw(n,\pi,b,w)$ if we now view the~$b-1$ gray balls 
as black.
\end{proof}

\section{Some properties of the~$\ul$ family}\label{sec:props}

In this section we derive some basic properties of the $\ul$ family. First we record some moment and tail bounds.
\begin{proposition}[Moment Bounds]
Fix~$w$ and~$(a_k)_{k\geq 1}$, let~$Z\sim\ul\bklr{w;(a_k)_{k\geq1}}$, 
and let~$c$ be the normalising constant from~\eq{2}, depending only on~$w$ and~$(a_k)_{k\geq 1}$.
Then for any positive integer~$m$,
\be{
\IE Z^m\leq \inf_{\{\ell: a_\ell>0\}} \frac{c}{\ell} \left(\frac{w a_\ell}{\ell}\right)^{-(w+m)/\ell}  \Gamma\left(\frac{w+m}{\ell}\right).
}
Moreover, $\ul\bklr{w;(a_k)_{k\geq1}}$ is uniquely determined by its moments.
\end{proposition}
\begin{proof}
If~$\ell$ is such that~$a_\ell>0$, we have
\bes{
\IE Z^m&=c \int_0^\rho x^{w+m-1}e^{-w \sum_{k\geq 1}\frac{a_k}{k} x^k} dx \\
	& \leq c \int_0^\rho x^{w+m-1}e^{-w \frac{a_\ell}{\ell} x^\ell} dx  \\
	& \leq c \int_0^\infty x^{w+m-1}e^{-w \frac{a_\ell}{\ell} x^\ell} dx \\
	&= \frac{c}{\ell} \left(\frac{w a_\ell}{\ell}\right)^{-(w+m)/\ell}  \Gamma\left(\frac{w+m}{\ell}\right),
}
which proves the first assertion. For the second, the bound above and Stirling's approximation shows that 
\be{
\limsup_{m\to\infty} \frac{\bklr{\IE Z^m}^{1/2m}}{m} <\infty,
}
and so in particular, Carleman's condition for the Stieljes moment problem is satisfied.
\end{proof}

\begin{proposition}[Mills Ratio Tail Bound]
Fix~$w$ and~$(a_k)_{k\geq 1}$, and let~$Z\sim\ul\bklr{w;(a_k)_{k\geq1}}$. 
For each~$\alpha>0$, there is a constant~$C_\alpha$ such that for~$x>\alpha$, 
\be{
P(Z\geq x)\leq C_\alpha u(x).
}
\end{proposition}
\begin{proof}
We show that~$\frac{P(Z\geq x)}{u(x)}$ is non-increasing in~$x$, from which the proposition follows with
$C_\alpha:=\frac{P(Z\geq \alpha)}{u(\alpha)}.$ 
Note that~$u(x)= c e^{-B(x)}$, where we define
\be{
B(x):=-(w-1)\log(x)+ \sum_{k\geq 1} \frac{a_k}{k} x^k.
}
Note that~$B''(x)\geq0$ so~$B'$ is non-decreasing.
Then
\be{
\frac{d}{dx} \left(\frac{P(Z\geq x)}{u(x)} \right) = B'(x) e^{B(x)} \int_x^\rho e^{-B(y)} dy -1 \leq e^{B(x)} \int_x^\rho B'(y) e^{-B(y)} dy -1 =0.\qedhere
} 
\end{proof}

In Theorem~\ref{THM2} we showed that if~$b>1$, then the limiting distribution of our urn model can be expressed as~$\ul\bklr{w;(a_k)_{k\geq 1}}$ multiplied by a beta random variable. It is natural to ask if such distributions are again in the~$\ul$ family. Our next examples show that this is not true in general, not even
for the limits appearing in Theorem~\ref{THM2}.

\begin{example}
If~$U\sim\Beta(1,1)$ and~$X\sim \Exp(1)$, then~$\law(UX)$ is not in the~$\ul$ family. Indeed, the density of~$UX$ for~$x>0$ is~$\int_x^\infty e^{-t}/t dt$, which goes to infinity like~$-\log(x)$ as~$x\to 0$, and hence is not in the~$\ul$ class.
\end{example}

\newcommand{\Erfc}{\mathop{\mathrm{erfc}}}
\begin{example} Let $\pi = \delta_1$ be the point mass at $1$; as discussed in Section~\ref{sec:2a}, the 
scaled limit of~$\urnlaw(n,{{\delta_1}},1,w)$ has density proportional to
\ben{\label{200}
 x^{w-1} \exp\{-C x^{2}\} dx
}
for some constant $C$. By Theorem~\ref{THM2}, the scaled limit of $\urnlaw(n,{{\delta_1}},2,1)$ has distribution~$\law(BZ)$, where $Z$ has density proportional to~\eq{200} with $w=2$, and where $B\sim \Beta(1,1)$ is independent of~$Z$. Using the density formula for products of independent random variables, we obtain that $BZ$ has density proportional to
\be{
  \int_{x}^\infty e^{-Cy^2}dy
}
which, up to scaling and multiplicative constants, is known as the \emph{complementary error function~$\Erfc(x)$}. If $BZ\sim\ul(v,(a_k')_{k\geq1})$ for some positive integer~$v$ and positive sequence $(a_k')_{k\geq1}$, then since
$\lim_{x\to0}  \int_{x}^\infty e^{-Cy^2}dy>0$, we must have $v=1$.
In this case, $a_k'$ are just the coefficients in the Taylor series expansion about zero of $-\log( \int_{x}^\infty e^{-Cy^2}dy)$, but since
\be{
 - \frac{\partial^4}{\partial x^4}\log\Erfc(x)\bigg|_{x=0} = 
  \frac{32(3-\pi)}{\pi ^2}<0,
}
we would have $a_4'= \frac{32(3-\pi)}{4!\pi ^2}<0$ in representation \eq{2}, 
so that $BZ$ cannot be in the $\ul$ family. 
\end{example}

\section{Acknowledgements}
AR is supported by NUS Research Grant R-155-000-167-112.
EP, AR, and NR are supported by ARC grant DP150101459.
This work was done partially while the authors were visiting the Institute for Mathematical Sciences, National University of Singapore in 2015 and 2016, supported in part by the Institute. We thank a referee for their careful reading and constructive comments.

\setlength{\bibsep}{0.5ex}
\def\bibfont{\small}

%
%

\end{document}